\documentclass[12pt]{amsart} 
\usepackage{parskip} 
\usepackage[latin1]{inputenc} 
\usepackage{color} 
\usepackage{geometry}
\usepackage{verbatim} 
\newtheorem{theorem}{Theorem}[section] 
\newtheorem{lemma}[theorem]{Lemma} 
\newtheorem{prop}[theorem]{Proposition} 

\theoremstyle{definition}

\newtheorem{remark}[theorem]{Remark} 

\providecommand{\Ric}{\mathop{\rm Ric}\nolimits}

\begin{document} 
	
	\title[MCF of graphs in GRW spacetimes]{Mean curvature flow of graphs in Generalized Robertson-Walker spacetimes with perpendicular Neumann boundary condition
	}
	
	
	\author{Jorge H. S. de Lira and Fernanda Roing 
	} 
	
	
	{Departamento de Matem\'atica, Universidade Federal do Cear\'a \\CEP 60455-900 
		Brazil \\ 
		
	} 
	
	\date{Received: date / Accepted: date} 
	\begin{abstract} 
		We prove the longtime existence for the mean curvature flow problem with a perpendicular Neumann boundary condition in a Generalized Robertson-Walker (GRW) spacetime that obeys the null convergence condition. In addition, we prove that the metric of such a solution is conformal to the one of the leaf of the GRW in asymptotic time. Furthermore, if the initial hypersurface is mean convex, then the evolving hypersurfaces remain mean convex during the flow.
	\end{abstract} 
	
	\maketitle 
	\section{Introduction} 
	The mean curvature flow (MCF) has been extensively studied by many authors in Riemannian ambients. G. Huisken \cite{H:84} proved that any compact convex hypersurface in $\mathbb{R}^n$ contracts to a  point as the flow decreases area. In contrast, when a hypersurface in a Lorentzian ambient moves by the MCF, it increases area and does not degenerate into a singularity. K. Ecker and G. Huisken were one of the pioneers in the study of the MCF in Lorentzian ambients. In \cite{EH:91} they constructed spacelike slices of prescribed mean curvature making use of the MCF. Later, K. Ecker dealt with the MCF of spacelike hypersurfaces in asymptotically flat spacetimes in \cite{E:93}. In the sequence, he investigated in \cite {E:97} longtime solutions for the MCF of noncompact spacelike hypersurfaces in Minkowski space. In this same work, he also studied the flow generated by a spacelike graph with Dirichlet boundary condition, and showed that it exists for all time and converges to a maximal hypersurface. In \cite{L:14} B. Lambert dealt with the MCF in the Minkowski space with a perpendicular Neumann boundary condition. He proved the longtime existence of the flow and that it converges to a homotetically expanding hyperbolic solution (a soliton in the Minkowski space). 
	
	Our aim here is to investigate the mean curvature flow of spacelike graphs in the so-called Generalized Robertson-Walker spacetimes (GRW), a special class of Lorentzian manifolds endowed with a closed conformal timelike vector field. Those special solutions of Einstein's field equations have been studied from both mathematical and physical points of view (see, e.g., \cite{ARS:95}, \cite{ARS:97}, \cite{M:73} and \cite{O'N:83}). In order to define them, we consider a connected oriented Riemannian manifold $(M^n, \sigma)$, an open interval $I=(s_-, s_+)$ of the extended real line $\mathbb{R}\cup \{-\infty, +\infty\}$ and a positive smooth function $\rho :I\to \mathbb{R}$. The Generalized Robertson-Walker spacetime modelled upon these data is the Lorentzian product manifold $N^{n+1}=I\times M^n$ endowed with the metric represented by 
	\begin{equation}
	\label{warped-m}
	\langle \cdot , \cdot \rangle = - \pi_\mathbb{R}^ * {\rm d}s^2 + \rho^2(\pi _\mathbb{R}) \pi _M^ * \sigma, 
	\end{equation}
	where $s$ is the natural coordinate in $I\subset \mathbb{R}$ and $\pi _ \mathbb{R}:N\to \mathbb{R}$ and $\pi _ M :N\to M $ denote the canonical projections. Hence, GRW is a \emph{warped} product space that we indicate by $-I \times_\rho M^n$. 
	Some well known spaces that are examples of GRW spacetimes are the {\it Minkowski space} (for $M^n =\mathbb{H}^n$ and $\rho(s) =s$ or $M^n = \mathbb{R}^n$ and $\rho(s)=1$), the {\it deSitter space} (for $M^n = \mathbb{S}^n$ and $\rho(s) = \cosh s$), the {\it anti deSitter space} (for $M^n = \mathbb{H}^n$ and $\rho(s) = \cosh s$), the {\it steady state spacetime}, (for $M^n = \mathbb{R}^n$ and $\rho(s) = e^s$) and {\em Einstein deSitter space} (for $I=(0, \infty)$, $M^ 3=\mathbb{R}^ 3$ and $\rho (s) =s^ {2/3}$) among others. 
	
	Along this paper, we will assume a natural energy condition on the GRW spacetime, the so-called {\em null convergence condition} (NCC). Namely, we say that the NCC holds on $N = -I\times _ \rho M$ provided its Ricci tensor satisfies 
	$$\overline{\Ric }(Y, Y)\geq 0$$ 
	for every light-like vector field $Y$ ($\langle Y, Y\rangle =0$).
	This constraint on the Ricci tensor can be physically interpreted as a necessary condition in order to $N$ obeys the Einstein's fields equations.
	
	Taking [\ref{O'N:83}, Proposition 7.35] into account, we have that the vector field $X=\rho(s) \frac{\partial}{\partial s}$ satisfies
	\begin{eqnarray}\label{conformal} 
	\bar{ \nabla }_Y X=\rho '(s)Y, 
	\end{eqnarray} 
	for every vector field  $Y\in \Gamma(N)$, where $\bar \nabla$ denotes the Levi-Civit\`a connection of $N$ and $'$ indicates the derivative with respect to $s$. It means that $X$ is a closed conformal vector field with conformal factor $\rho '$ (see \cite{M:99}).

	In what follows, we consider the mean curvature flow of graphs in $N $.
	Given $T>0$ and a compact convex domain $\bar \Omega \subset M$  with smooth boundary, let $u: \bar\Omega \times [0, T)\to \mathbb{R}$ be a smooth function. This defines a one-parameter  family of submanifolds $\Sigma_t$, $t\in [0,T)$, parametrized as graphs  by  the isometric immersions
	\begin{eqnarray}\label{m} 
	\varphi(x,t) = \varphi _t(x)=
	(u(x, t), x) \in I\times M, \quad t\in [0, T), \,\, x\in \bar\Omega.
	\end{eqnarray} 
	These graphs evolve by the mean curvature flow if the map $\varphi: [0,T) \times \bar\Omega \to N$ satisfies
	\begin{equation}
	\label{mcf-1}
	\frac{\partial\varphi}{\partial t} = H\nu
	\end{equation}
	where $\nu= \nu(x,t)$ is the timelike unit future direct normal to $\Sigma_t$ at the point $\varphi_t(x)$ and
	\[
	H = H(x,t) = \operatorname{div} \nu (x,t)
	\]
	is the (scalar) mean curvature of $\Sigma_t$ at that point. Now we define both initial and boundary conditions for this flow. For that, we consider the
	{\it conformal cylinder} over the boundary $\partial \bar \Omega $  defined by
	\begin{eqnarray*}\label{cylinder} 
		K= \partial\bar \Omega \times I \subset N. 
	\end{eqnarray*} 
	Let $\mu $ be the normal to $K$ chosen in a away that its second fundamental form is non-negative definite. Then we impose the following \emph{Neumann boundary condition} to the mean curvature flow (\ref{mcf-1})
	\begin{equation}
	\label{neumann}
	\langle \nu, \mu\rangle = 0 \quad \mbox{ on } \partial\bar \Omega \times [0, T)
	\end{equation}
	This means that, for each $t\in [0, T)$,  the graph $\Sigma_t$ is perpendicular to $K$ along the boundary $\partial\Sigma_t \subset K$. The initial condition of (\ref{mcf-1}) is given by the graph $\Sigma_0$ of the function $u_0:=u(\cdot\,, 0)$. In sum, we have posed an initial value problem with orthogonal Neumann condition for the mean curvature flow (\ref{mcf-1}) that can be read in terms of the time-dependent function $u$ as follows
	\begin{eqnarray}\label{Fn} 
	\begin{cases}\label{F2} 
	\frac{\partial u}{\partial t}(x, t)= \frac{W}{\rho} \operatorname{div} \frac{Du}{\rho W} + \frac{\rho'}{\rho} \Big( n+\frac{|Du|^2}{\rho^2}\Big) \,\, \mbox{ in }\,\, \bar \Omega \times [0, T) & \\ 
	u (\cdot\,, 0)=u_0, & \\ 
	\langle \nu, \mu\rangle=0 \,\, \mbox{ on }\,\, \partial \bar \Omega \times [0, T) 
	\end{cases} 
	\end{eqnarray} 
	where $Du$ denotes the gradient of the function $u_t:=u(\cdot,\, t)$ and $W=\sqrt{\rho ^2 - |Du|^2}$.

	Now, we are in position to enunciate our result:
	
	\begin{theorem}\label{theo2} 
		Let $-I\times _ \rho M ^n$ be a GRW spacetime that obeys the null convergence condition and whose warping function satisfies 
		\begin{equation}
		\label{max-princ}
		\rho'(s) \ge 0 \,\, \mbox{ and } \,\, \Big(\frac{\rho'(s)}{\rho(s)}\Big) '  \le 0 \,\, \mbox{with } \,\, \Big(\frac{\rho'(\tilde{s})}{\rho(\tilde{s})}\Big) '< 0 \,\, \mbox{for some} \,\, \tilde{s}>\max _{\Sigma _0}s
		\end{equation}
		and
		\begin{equation}
		\label{HM}
		C_- :=\limsup_{s\to s_-} \frac{\rho' (s)}{\rho(s)} <+\infty \quad \mbox{ and } \quad  C_+ :=\liminf_{s\to s_+} \frac{\rho' (s)}{\rho(s)} >0.
		\end{equation}
		Let $\bar \Omega \subset M^n$ be a compact convex domain, and $K$ its respective conformal cylinder. If $\Sigma _0^n$ is a spacelike graph over $\Omega $ that intersects $K$ orthogonally, then there exist a longtime solution to the problem \eqref{Fn} whose metric is conformal to the one of the leaf in asymptotic time. In addition, if $\Sigma _0^n$ is mean convex (possesses positive mean curvature) and $\overline{Ric }(\partial s, \partial s)$ is bounded from above, then the solution remains mean convex for all times.
	\end{theorem} 
	
	\begin{remark}
		We reinforce the reasonableness of the hypothesis \eqref{HM}. We will see in posterior computations \eqref{heat-s} that $(\frac{\partial }{\partial _t}-\Delta )s$ is positive, and as a consequence of the maximum principle that s is uniformly bounded from below, therefore $\frac{\rho '(s)}{\rho (s)}$ being bounded away from infinity will be actually a straight consequence under the assumption \eqref{max-princ}. 
		
		Remark that the mean curvature of the slices $\{s\}\times M $ are given by $H=n\frac{\rho '(s)}{\rho (s)}$ (see [\ref{O'N:83}, Proposition 7.35]). In this case the imposition involving $C_+$ means that the mean curvature of the slices are uniformly bounded away from zero. From a phenomena known as avoidance principle (see \cite{B:16}) which asserts that if two compact manifolds are initially disjoint, then they will keep so during the mean curvature flow, if we had a maximal slice (with zero mean curvature, whose flow is stationary), then our flow would also stop in height. Henceforth the imposition involving $C_+$ may be interpreted as the flow never stopping in height.
	\end{remark}
	
	If we rule the Minkowski as $-\mathbb{R}\times _s \mathbb{H}^n$, then the longtime existence of [\ref{L:14}, Theorem 1.3] as well as the the preservation of the mean convexity and the asymptotic behavior are corollary of the above result. Our technique relies on the employment of the weak maximum principle in order to obtain {\it a priori} estimates that jointly with the classical parabolic equations theory will assure the longtime solvability of the problem \eqref{Fn}. Our paper is organized as follows: in Section \ref{nonpar} we describe the problem \eqref{Fn} in a non-parametric setting. In Section \ref{equations} we derive the evolution of some important quantities as well as its derivatives in the boundary. In Section \ref{estimates} we obtain {\it a priori} $C^0$, $C^1$ and second order estimates concerning a solution of the problem \eqref{Fn}. In Section \ref{proof} we summirize the proof of the Theorem \ref{theo2}.
	
	\section{A non-parametric formulation of the problem}\label{nonpar}
	
	We have adopted from the beginning a graphical parametrization of the evolving hypersurfaces $\Sigma_t$, $t\in [0, T)$. This reduces the geometric mean curvature flow (\ref{mcf-1}) to a parabolic quasilinear PDE as we will deduce in the sequel. 
	
	The induced metric $g$ (Riemannian, since $\Sigma _t$ is supposed to be spacelike) in $\Sigma_t$ is expressed in terms of local coordinates $\{x^k\}_{k=1}^n$ in $\bar \Omega$ as
	\[
	g_{ij}=-u_iu_j+\rho^2 (u)\sigma _{ij},
	\] 
	where $\sigma_{ij}$ are the local components of the Riemannian metric $\sigma$ in $M$ and $u_i = \frac{\partial u}{\partial x^i}$. Note that $s=u(\cdot\,, t)$ along $\Sigma_t$. The timelike future pointing unit normal vector field along the graph $\Sigma_t$ is written in terms of the function $u$ and its gradient as
	\begin{eqnarray}\label{nu} 
	\nu =\frac{1}{\rho(u)}\frac{1}{\sqrt{\rho (u)^2-|Du|^2}}\bigg(\rho^2(u) \frac{\partial}{\partial s}+\sigma^{ij}\frac{\partial u}{\partial x^i} \frac{\partial}{\partial x^j}\bigg), 
	\end{eqnarray} 
	where $Du=\sigma^{ij}\frac{\partial u}{\partial x^i} \frac{\partial}{\partial x^j}$ denotes the gradient of the function $u_t:=u(\cdot,\, t)$. The fact that $\Sigma_t$ is a spacelike hypersurface, i.e., that the induced metric is Riemannian is equivalent to
	\[
	|Du(\cdot\,, t)|< \rho (u(\cdot\,, t)).
	\]
	We say that  $\Sigma_t$ is {\it strictly spacelike} if there exists a positive constant $c$ such that $|Du(\cdot\,, t)|+c< \rho (u(\cdot\,, t))$. 
	The mean curvature of $\Sigma_t$ associated to \eqref{nu} is given by 
	\begin{eqnarray*}\label{a} 
		H={\rm div}  \frac{Du}{\rho (u) W} +\frac{\rho '(u)}{W}\left(n+\frac{|Du|^2}{\rho (u)^2}\right), 
	\end{eqnarray*} 
	where $W=\sqrt{\rho ^2 - |Du|^2}$. Indeed one has
	\begin{equation}\label{nabla-frame}
	\bar\nabla_{\frac{\partial}{\partial x^i}} \frac{\partial}{\partial s} = \frac{\rho'(s)}{\rho(s)} \frac{\partial}{\partial x^i} \quad \mbox{ and } \quad \bar\nabla_{\frac{\partial}{\partial x^i}} \frac{\partial}{\partial x^j} = \rho(s)\rho'(s) \sigma_{ij }\frac{\partial}{\partial s} + \Gamma^k_{ij} \frac{\partial }{\partial x^k},
	\end{equation}
	where $\Gamma^k_{ij}$ are the Christoffel symbols of the Riemannian connection  relative to the local coordinates $\{x^k\}_{k=1}^n$ in $(M, \sigma)$. The second fundamental form of  $\Sigma_t\subset N$ with respect to the normal vector field $\nu$ is given by
	\[
	II(Y, Z)=\langle \bar \nabla_Y\nu , Z\rangle, 
	\]
	for all vector fields $Y, Z$ tangent to $\Sigma_t$. Therefore, the local components of this second fundamental form are given by
	\[
	a_{ij} =\frac{1}{\rho W} \big(\rho^2 u_{i;j} - 2\rho \rho' u_i u_j + \rho^3 \rho'\sigma_{ij}\big),
	\]
	where  $u_{i;j} = \partial_i u_j - \Gamma^k_{ij} u_k$ are the local components of the Hessian of $u$ in $(M, \sigma)$. Taking traces of the second fundamental form with respect to the induced metric, one computes the mean curvature of $\Sigma_t$, that is,
	\begin{align*}
		H& =\frac{1}{\rho^2}\bigg( \sigma^{ij}+\frac{1}{W^2} u^i u^j\bigg) a_{ij}\\
		& =\frac{\rho}{W} g^{ij} u_{i;j} + \frac{1}{\rho^3 W}\bigg(-2 \rho^3\rho' \frac{|D u|^2}{W^2}+ \rho^3 \rho' \bigg(n+\frac{|Du|^2}{W^2}\bigg)\bigg)\\
		& \,\, =\frac{1}{\rho} \bigg(\frac{1}{W} \Delta u + \frac{1}{W^3}\langle D_{Du} Du, Du\rangle\bigg) - 2 \frac{\rho'}{W} \frac{|D u|^2}{W^2} + \frac{\rho'}{W} \bigg(n+\frac{|Du|^2}{W^2}\bigg)\\
		& \,\, = \frac{1}{\rho} \bigg(\operatorname{div} \frac{Du}{W}+\frac{\rho\rho'}{W} \frac{|Du|^2}{W^2}\bigg)  - 2 \frac{\rho'}{W} \frac{|D u|^2}{W^2} + \frac{\rho'}{W} \bigg(n+\frac{|Du|^2}{W^2}\bigg)\\
		& \,\, = \frac{1}{\rho} \operatorname{div} \frac{Du}{W}   + n\frac{\rho'}{W} \\
		& = \operatorname{div}\frac{Du}{\rho W} + \frac{\rho'}{\rho^2} \frac{|Du|^2}{W}+ n\frac{\rho'}{W}\cdot
	\end{align*}
	A combination of (\ref{m}), (\ref{mcf-1}) and (\ref{nu}) yields
	\begin{eqnarray}\label{aga}
	H = -\bigg\langle \frac{\partial\varphi}{\partial t}, \nu\bigg\rangle =-\bigg\langle \frac{\partial u}{\partial t} \frac{\partial}{\partial s}, \nu\bigg\rangle =
	\frac{\rho (u)}{W}\frac{\partial u}{\partial t}\cdot
	\end{eqnarray} 
	We conclude that the non-parametric formulation of (\ref{mcf-1}) is given by the quasilinear parabolic PDE 
	\begin{equation}
	\label{mcf-eqt}
	\frac{\partial u}{\partial t} =\frac{W}{\rho} \operatorname{div} \frac{Du}{\rho W} + \frac{\rho'}{\rho} \bigg( n+\frac{|Du|^2}{\rho^2}\bigg). 
	\end{equation}
	It follows easily that the  initial  and boundary (Neumann) conditions are expressed in terms of the function $u$ as in the second and third lines in (\ref{Fn}), respectively.
	%
	%
	%
	%
	%

	We observe that the conformal vector field $X$ generates a one parameter family of conformal maps. Indeed the warped metric (\ref{warped-m}) can be rewritten as
	\[
	\rho^2(s) \bigg(-\frac{ds^2}{\rho^2(s)}+ \sigma\bigg) =:\lambda^2(\varsigma) (-d\varsigma^2 + \sigma),
	\]
	where $\varsigma \in J$ is the parameter of the conformal flow $\Phi_\varsigma:=\Phi(\varsigma, \cdot\,): N\to N$
	generated by $X$. This parameter is given by
	\[
	\varsigma(s)= \int_0^s \frac{dr}{\rho(r)}
	\]
	and $J = \varsigma(I)$.
	The conformal factor $\lambda=\lambda(\varsigma)$ depends only on this parameter and it is related to the warping function $\rho=\rho(s)$ by
	\[
	\lambda(\varsigma(s)) = \rho(s).
	\]
	Hence
	\[
	\frac{d\lambda}{d\varsigma}\varsigma'(s) = \rho'(s),
	\]
	that is,
	\begin{equation*}
		\frac{d\lambda}{d\varsigma}(\varsigma(s)) = \rho(s) \rho'(s).
	\end{equation*}
	We can parametrize the evolving graphs $\Sigma_t$, $t\in [0, T)$, in terms of the parameter $\varsigma$ defining the time-dependent function $z=\varsigma(x,t)$. Hence
	\[
	\frac{\partial u}{\partial t} = \frac{ds}{d\varsigma} \frac{\partial z}{\partial t} =\rho(u(x,t)) \frac{\partial z}{\partial t} 
	\]
	and
	\[
	Du = \rho Dz \quad\mbox{ and } \quad W = \rho \sqrt{1-|Dz|^2} =: \rho \mathcal{W}. 
	\]
	Hence the prescribed mean curvature equation is written in terms of $z$ as
	\begin{equation*}
		\label{z}
		\operatorname{div} \frac{Dz}{\mathcal{W}} + n \frac{\rho ' }{\rho}\frac{1}{\mathcal{W}} - \rho H =0
	\end{equation*}
	and the non-parametric mean curvature flow equation \eqref{mcf-eqt} becomes
	\begin{equation}
	\label{mcf-z}
	\rho \frac{\partial z}{\partial t} = \frac{\mathcal{W}}{\rho} \operatorname{div} \frac{Dz}{\mathcal{W}} + n \frac{\rho ' }{\rho}\cdot
	\end{equation}
	In local coordinates this equation becomes
	\begin{equation*}
		\frac{\partial z}{\partial t} = \frac{1}{\rho^2} \bigg( \sigma^{ij} + \frac{1}{\mathcal{W}^2} z^i z^j \bigg) z_{i;j} +  n \frac{\rho ' }{\rho^2},
	\end{equation*}
	that is,
	\begin{equation*}
		\frac{\partial z}{\partial t} = \frac{1}{\rho}H\mathcal{W}.
	\end{equation*}
	We conclude from \eqref{mcf-z} that the parabolic maximum principle holds for the mean curvature flow in this setting if 
	\begin{equation*}
		\rho'(s) \ge 0 \quad \mbox{ and } \quad \Big(\frac{\rho'(s)}{\rho(s)}\Big) '  \le 0,
	\end{equation*}
	justifying the hypothesis \eqref{max-princ}.
	
	
	\section{Fundamental equations}\label{equations}
	\subsection{Evolution equations by MCF}
	
	In this section we will compute the evolution equations of some  geometric  invariants of the evolving graphs $\Sigma_t$, $t\in [0, T)$. Firstly, we fix some notations. Let $\phi: I\to \mathbb{R}$ be a solution of the equation
	\begin{equation*}
		\phi(s) = \int^s_{s_0} \rho(r) {\rm d}r,
	\end{equation*} 
	for some fixed $s_0\in I$. We also consider the {\it support function}
	\begin{equation*}
		\label{Theta}
		\Theta =-\langle X,\nu\rangle.
	\end{equation*}
	It follows from (\ref{nu}) that 
	\begin{equation}
	\label{thetaW}
	\Theta = \frac{\rho^2}{W}\cdot
	\end{equation}
	This function can be also written in terms of the hyperbolic angle $\alpha $ between $X=\rho\frac{\partial}{\partial s}$ and $\nu$ as follows
	\begin{equation}
	\label{thetarho}
	\Theta = \rho \cosh\alpha .
	\end{equation}
	Hence, $\Theta$ can be used to estimate the gradient of the functions $u(\cdot,\, t), t\in [0, T)$. 
	In what follows, $\nabla$ denotes the Riemannian connection in $\Sigma_t$. Hence we have
	\begin{equation}\label{nabla-s}
	\bar\nabla s = -\frac{\partial}{\partial s}
	\end{equation}
	and
	\begin{equation*}
		\nabla s = -\Big(\frac{\partial}{\partial s}\Big)^\top = -\frac{\partial }{\partial s} - \Big\langle \frac{\partial}{\partial s}, \nu\Big\rangle \nu = -\frac{\partial}{\partial s} +\frac{\Theta }{\rho}\nu,
	\end{equation*} 
	where $(\cdot )^\top $ stands for the orthogonal  tangential  projection onto $\Sigma_t$.

	\begin{lemma}\label{ev} Let $\Delta$ be the Laplace-Beltrami operator in $\Sigma_t$ with respect to the time-dependent induced metric $g$ and let
		$Q = \frac{\partial}{\partial t}-\Delta$. Then we have
		\begin{align} 
			& \label{heat-s} Qs = \frac{\rho'}{\rho} (n+|\nabla s|^2) \\
			& \label{heat-phi}  Q\phi =n\rho'\\
			& \label{heat-Theta} Q\Theta =2\rho' H-\Theta \bigg( {\rm \overline{Ric}}(\nu,\nu)+|A|^2+n\frac{\rho''}{\rho}\bigg)\\
			& \label{heat-H} QH=-(|A|^2+ {\rm \overline{Ric}}(\nu, \nu )) H \\
			& \label{heat-kappa}
			Q\kappa = C_0n + C_0\frac{\rho }{\rho '} \frac{\rho '' }{\rho ' } |\nabla s|^2,
		\end{align} 
		where $\kappa(s) = C_0\int^s_{s_0} \frac{\rho(r)}{\rho'(r)} dr,$ for $s_0\in I$ is arbitrarily fixed and $C_0>0$ a constant to be fixed later.
	\end{lemma} 
	\begin{proof} 
		First of all, we observe that
		\begin{align}\label{ds}
			\frac{\partial s}{\partial t}=\Big\langle \bar\nabla s , \frac{\partial\varphi}{\partial t}\Big\rangle = -H\Big\langle \frac{\partial}{\partial  s}, \nu \Big\rangle . 
		\end{align} 
		from what follows that
		\begin{align}\label{phii}
			\frac{\partial \phi }{\partial t} = \rho \frac{\partial s}{\partial t}=-H\langle X, \nu \rangle =H\Theta. 
		\end{align}
		Now, since 
		\begin{eqnarray}\label{phi} 
		\nabla \phi =\rho (s)\nabla s=-X^\top , 
		\end{eqnarray} 
		using an orthonormal local frame $\{{\sf e}_k\}_{k=1}^n$ in $\Sigma_t$ and equation \eqref{conformal} yields
		\begin{align}\label{phiii} 
			\Delta \phi  =  -\sum_k \langle \bar\nabla_{{\sf e}_k} X^\top, {\sf e}_k\rangle =  -\sum_k \langle \bar\nabla_{{\sf e}_k} X, {\sf e}_k\rangle + \sum_k \langle \bar\nabla_{{\sf e}_k} \Theta \nu, {\sf e}_k\rangle=-n\rho' +H\Theta.
		\end{align} 
		Therefore joining \eqref{phii} and \eqref{phiii}, one gets (\ref{heat-phi}). These calculations also yields \eqref{heat-s}.
		
		Now, recall the evolution equation of the normal vector field $\nu$ under the mean curvature flow (\ref{mcf-1}). We have
		\begin{equation}
		\label{evol-nu}
		\Big\langle \bar\nabla_{\frac{\partial \varphi}{\partial t}} \nu, \frac{\partial \varphi}{\partial x^i}\Big\rangle = - \Big\langle \nu, \bar\nabla_{\frac{\partial \varphi}{\partial x^i}}\frac{\partial \varphi}{\partial t}\Big\rangle
		=- \Big\langle \nu, \bar\nabla_{\frac{\partial \varphi}{\partial x^i}}H\nu\Big\rangle  = \frac{\partial H}{\partial x^i}\cdot
		\end{equation}
		Therefore, using this expression and (\ref{conformal}) one obtains
		\begin{align}
			\label{dt-Theta}
			\frac{\partial \Theta}{\partial t} = -\big\langle \bar\nabla_{\frac{\partial\varphi}{\partial t}} X, \nu\big\rangle -\big\langle X, \bar\nabla_{\frac{\partial\varphi}{\partial t}} \nu\big\rangle =\rho ' H-\langle \nabla H, X\rangle.
		\end{align}
		Now, we recall the evolution of the induced metric with respect to a flow $\psi$ of the form
		\begin{equation}
		\label{alt-0}
		\frac{\partial \psi}{\partial \varsigma} = X,
		\end{equation}
		that is
		\begin{equation}
		\label{alt}
		\frac{\partial \psi}{\partial \varsigma} = \Theta \nu +X^\top.
		\end{equation}
		We have
		\begin{align*} \frac{\partial}{\partial \varsigma} g_{ij} = \Big\langle \bar\nabla_{\frac{\partial}{\partial x^i}} \Theta \nu, \frac{\partial}{\partial x^j}\Big\rangle + \Big\langle \frac{\partial}{\partial x^j}, \bar\nabla_{\frac{\partial}{\partial x^j}} \Theta \nu\Big\rangle + \pounds_{X^\top} g_{ij} = 2a_{ij}\Theta + \pounds_{X^\top} g_{ij}.
		\end{align*}
		The normal vector field evolves according to the expression
		\begin{equation*}
			\label{evo-nu}
			\bar\nabla_{\frac{\partial \psi}{\partial \varsigma}} \nu = \nabla \Theta + AX^\top,
		\end{equation*}
		where $A$ denotes the Weingarten map of $\Sigma _t$, from what follows the evolution law of the second fundamental form
		\[
		\frac{\partial}{\partial \varsigma} a_{ij}= \Theta_{i;j} -\Big\langle\bar R\Big(\frac{\partial}{\partial x^i}, \nu\Big)\nu, \frac{\partial}{\partial x^j}\Big\rangle \Theta + a_{i}^k a_{kj}\Theta + \pounds_{X^\top} a_{ij}.
		\]
		We conclude that the mean curvature evolves along the flow (\ref{alt}) as
		\begin{align*}
			\frac{\partial H}{\partial \varsigma} = \frac{\partial g^{ij}}{\partial \varsigma} a_{ij} + g^{ij}\frac{\partial }{\partial \varsigma} a_{ij}= -2  a^{ij}a_{ij}\Theta + \Delta \Theta - \overline{\rm Ric}(\nu, \nu) \Theta + a^{ij}a_{ij}\Theta + \pounds_{X^\top} H.
		\end{align*}
		Therefore we obtain
		\begin{align*}
			\Delta \Theta = |A|^2\Theta +\overline{\rm Ric}(\nu, \nu) \Theta + \frac{\partial H}{\partial \varsigma}- \pounds_{X^\top} H.
		\end{align*}
		Combining this equation and (\ref{dt-Theta}) and observing that $\pounds_{X^\top} H = \langle X, \nabla H\rangle$  one concludes that 
		\begin{align*}
			\Big(\Delta -\frac{\partial}{\partial t}\Big) \Theta = |A|^2\Theta +\overline{\rm Ric}(\nu, \nu) \Theta - \rho' H + \frac{\partial H}{\partial \varsigma}\cdot
		\end{align*}
		We can check that
		\begin{equation*}
			\frac{\partial H}{\partial \varsigma} = X[H]= -\rho' H + n \frac{\rho ''}{\rho} \Theta. 
		\end{equation*}
		Indeed, given coordinate vector fields $\{\partial_i\}_{i=1}^n$ in $\Sigma_t$ we compute
		\begin{eqnarray*}
			X\langle\bar\nabla_{\partial_i} \partial_j, \nu\rangle & =& \langle \bar\nabla_X \bar\nabla_{\partial_i} \partial_j, \nu\rangle + \langle \bar\nabla_{\partial_i} \partial_j, \bar\nabla_{X} \nu\rangle\\
			& = & \langle \bar\nabla_{\partial_i} \bar\nabla_X \partial_j, \nu\rangle + \langle \bar\nabla_{[X, \partial_i]} \partial_j, \nu\rangle + \langle \bar R(X,\partial_i)\partial_j,\nu\rangle + \langle \bar\nabla_{\partial_i} \partial_j, \bar\nabla_{X} \nu\rangle \\
			&=& \langle \bar\nabla_{\partial_i} \bar\nabla_{\partial_j} X, \nu\rangle + \langle \bar R(\partial_i, X)\nu, \partial_j\rangle + \langle \bar\nabla_{\partial_i} \partial_j, \bar\nabla_{X} \nu\rangle + \langle \bar\nabla_{[X, \partial_i]} \partial_j, \nu\rangle \\
			& =& \rho'\langle \bar\nabla_{\partial_i} \partial_j, \nu\rangle + \langle \bar R(\partial_i, X)\nu, \partial_j\rangle + \langle \nabla_{\partial_i} \partial_j, \bar\nabla_{X} \nu\rangle + \langle \bar\nabla_{[X, \partial_i]} \partial_j, \nu\rangle .
		\end{eqnarray*}
		We may choose coordinates such that $\nabla_{\partial_i} \partial_j =0$ at the point we are computing. Moreover, since $X$ and $\partial_i$ are push-forwards of coordinate vector fields in $\Sigma_t \times J$ we have $[X, \partial_i]=0$. Hence
		\[
		X\langle\bar\nabla_{\partial_i} \partial_j, \nu\rangle =\rho'\langle \bar\nabla_{\partial_i} \partial_j, \nu\rangle + \langle \bar R(\partial_i, X)\nu, \partial_j\rangle.
		\]
		Then taking traces in the expression above we get
		\begin{eqnarray*}
			g^{ij}X[a_{ij}]   =\rho' H-\overline{\rm Ric}(X, \nu) .
		\end{eqnarray*}
		Note that
		\[
		\pounds_X a_{ij} = X[a_{ij}] -\langle [X, \partial_i], A\partial_j\rangle - \langle A\partial_i, [X, \partial_i]\rangle =X[a_{ij}].
		\]
		Moreover
		\begin{eqnarray*}
			& & \pounds_X g_{ij} = X[g_{ij}] -\langle [X, \partial_i], \partial_j\rangle - \langle \partial_i, [X,\partial_j]\rangle = \langle \bar\nabla_{\partial_i}X, \partial_j\rangle + \langle \partial_i, \bar\nabla_{\partial_i}X\rangle   %
		\end{eqnarray*}
		and
		\begin{eqnarray*}
			a_{ij}\pounds_X g^{ij}  =  -2\rho ' H.
		\end{eqnarray*}
		Therefore
		\begin{eqnarray*}
			X[H] =\pounds_X H =  g^{ij} \pounds_X a_{ij} + a_{ij} \pounds_X g^{ij} =   \rho' H-\overline{\rm Ric}(X, \nu) -2\rho ' H = -\rho ' H -  {\rm Ric}_{\bar M}(X, \nu).
		\end{eqnarray*}
		However
		\begin{equation*}
			\overline{\rm Ric} (X, \nu) = -n\frac{\rho''}{\rho} \langle X, \partial_s\rangle \langle \nu, \partial_s\rangle = -n\frac{\rho''}{\rho}\langle X, \nu\rangle.
		\end{equation*}
		Therefore
		\[
		X[nH] = -\rho ' H +n\frac{\rho ''}{\rho}\Theta, 
		\]
		as we claimed. We  refer the reader to  [\ref{BBC:75}, Proposition 3.1] for the proof of similar expressions in a more general context. 

		In the same way, using the flow (\ref{mcf-1}) one gets
		\begin{equation*}
			\Delta H = |A|^2 H +\overline{\rm Ric}(\nu, \nu) H + \frac{\partial H}{\partial t}
		\end{equation*}
		from what follows (\ref{heat-H}).
		
		Let us compute \eqref{heat-kappa}. We have from \eqref{heat-s} that
		\begin{align*}
			Q\kappa &=\kappa' Qs - \kappa''|\nabla s|^2 = \kappa' \frac{\rho'}{\rho} (n+|\nabla s|^2) - \kappa''|\nabla s|^2 \\
			&  = n  \kappa' \frac{\rho'}{\rho} + \kappa' \Big( \frac{\rho'}{\rho}-\frac{\kappa ''}{\kappa '} \Big)|\nabla s|^2 .
		\end{align*}
		We conclude that
		\begin{equation*}
			Q\kappa = C_0n + C_0\frac{\rho }{\rho '} \frac{\rho '' }{\rho ' } |\nabla s|^2.
		\end{equation*}
	\end{proof} 
	
	\subsection{Derivatives on the boundary} 
	We will study the boundary derivatives of some suitable functions in order to employ Lemma \ref{wmp} in posterior computations. 
	
	\begin{lemma}\label{boundary}  Let $II^K$ be  the second fundamental form of the conformal cylinder $K\subset N$ with respect to the outward normal vector field $\mu$  defined by
		\[
		II^K(U, V)=\langle \bar\nabla _U\mu , V\rangle
		\] 
		for all vector fields $U,V $ tangent to $K$. Then, in any point of $K$ and for any vector field $U$ tangent to $K$ we have
		\begin{align} 
			& \label{nablarho} \langle \nabla \rho ,\mu \rangle =0\\
			& \label{IIK} II(\mu , U)=-II^K(\nu, U) \\
			& \label{nablatheta} \langle \nabla \Theta , \mu \rangle = -\Theta II^K (\nu, \nu )\\ 
			& \label{nablaH2} \langle \nabla H, \mu \rangle =-H II^K(\nu, \nu ).\\
			& \label{nablakappa}
			\langle \nabla \kappa, \mu\rangle =0
		\end{align} 
	\end{lemma} 
	
	\begin{remark}
		\label{rem-conv}
		The convexity of $K$ implies that $II^K(\nu, \nu)\ge 0$.
	\end{remark}
	
	\begin{remark}
		Henceforth we will keep the notation $Q = \frac{\partial}{\partial t}-\Delta$.
	\end{remark}
	
	\begin{proof} It follows from  \eqref{nabla-s} that
		\[
		\bar\nabla \rho = \rho' (s) \bar\nabla s = -\frac{\rho' (s) }{\rho(s)}\rho(s)\frac{\partial}{\partial s} = -\frac{\rho' (s) }{\rho(s)} X. 
		\]
		Hence, \eqref{nablarho} follows from the fact that $X$ is tangent to the cylinder $K$. 
		
		Now, in order to prove (\ref{IIK}) we first note that the Neumann condition can be written as $\langle \nu, \mu\rangle = 0$. Using this, one has
		\[
		II(\mu, U) = \langle \bar\nabla_{U} \nu, \mu\rangle = U\langle \nu, \mu\rangle - \langle \nu, \nabla_U \mu\rangle = -II^K(\nu, U). 
		\]

		The proof of (\ref{nablatheta}) combines  \eqref{conformal}, \eqref{IIK} and again the Neumann condition, as follows:
		\begin{align*} 
			\langle \nabla \Theta , \mu \rangle & =  -\langle \bar\nabla _ \mu X, \nu \rangle - \langle X , \bar\nabla _ \mu \nu \rangle \\
			& =  -\rho '(s) \langle \mu, \nu \rangle -II(\mu, X^\top)\\
			& =  II^K(\nu, X^\top)\\ 
			& = II^K(\nu, X) -\Theta II^K(\nu, \nu) \\
			& =  -\Theta II^K(\nu, \nu ),
		\end{align*} 
		where we also used the fact that $X$ is a principal direction in $K$ corresponding to a null principal curvature. 
		
		
		Using \eqref{mcf-1} and \eqref{evol-nu} one computes 
		\begin{eqnarray*} 
			0 =  \frac{d}{dt}\langle \mu, \nu \rangle =  \langle \bar \nabla _ {H\nu}\mu, \nu \rangle + \langle \mu, \nabla H \rangle  =  H II^K(\nu, \nu) +\langle \mu, \nabla H\rangle , 
		\end{eqnarray*} 
		from what we easily deduce \eqref{nablaH2}.
		
		FInally, \eqref{nablakappa} follows from $\langle \nabla s, \mu\rangle=0$.
	\end{proof} 
	
	\section{Maximum Principles and a priori estimates}\label{estimates} 
	
	Along this section we will obtain the necessary estimates relying upon the following lemmas:
	
	\begin{lemma}\label{wmp}
		Let $\bar \Omega $ be a compact domain with smooth boundary and let $f:\bar \Omega \times [0,T)\to \mathbb{R}$ be a function twice differentiable on space and once in time. 
		Let $Q=\Big( \frac{\partial }{\partial t}-\Delta \Big) $, where $\Delta $ is the time-dependent Laplace-Beltrami operator. Suppose that
		$$
		\begin{cases} 
		Qf(x, t)\leq 0 \, \mbox{ for all } \, (x, t)\in \bar \Omega \times [0, T)\,  \mbox{ such that } \,\nabla f(x)=0, & \\ 
		\langle \nabla f, \mu \rangle \leq 0 \,\mbox{ for all } \, (x, t)\in \partial \bar \Omega \times [0,T). & 
		\end{cases} $$
		Then $\sup _{\bar\Omega \times [0,T)}f\leq \max _{\bar \Omega }f(x, 0) $.
		In addition, if
		$$\begin{cases} 
		Qf(x, t)\geq 0 \, \mbox{ for all } \, (x, t)\in \bar \Omega \times [0, T)\,  \mbox{ such that } \,\nabla f(x)=0, & \\ 
		\langle \nabla f, \mu \rangle \geq 0 \,\mbox{ for all } \, (x, t)\in \partial \bar \Omega \times [0,T), & 
		\end{cases} $$
		then  $\inf _{\bar\Omega \times [0,T)}f\geq \min _{\bar \Omega }f(x, 0) $.
	\end{lemma}
	\begin{proof}
		This proof is an adaptation of Theorems 5.3.2 and 5.3.3 of \cite{H:11}. Suppose that we have an interior point $(x_0, t_0)\in {\rm int}\bar \Omega   \times(0, T)$ where the function $\tilde{f}=e^ {-\epsilon t}f$, for $\epsilon >0$ attains a positive maximum. In this case we have that $\nabla \tilde{f}(x_0, t_0)=0$ and that the Hessian matrix $\nabla ^ 2 \tilde{f}(x_0, t_0)$ is non-positive definite, hence $\Delta \tilde{f} (x_0, t_0)\le  0$. From $\nabla \tilde{f}=e^ {-\epsilon t}\nabla f$ and hypothesis
		$$ \frac{\partial \tilde{f}}{\partial t}(x_0, t_0)\leq Q\tilde{f}(x_0, t_0)=-\epsilon \tilde{f}(x_0, t_0)+e^ {-\epsilon t}Qf (x_0, t_0)\leq -\epsilon \tilde{f}(x_0, t_0)< 0,$$
		therefore $\tilde{f}$ is decreasing in time at any interior positive maximum point.
		
		Now suppose that we have a positive maximum point $(x_0, t_0) \in \partial \bar \Omega \times (0, T)$ non-decreasing in time. As above,
		\[
		Q\tilde{f}(x_0, t_0)\leq -\epsilon \tilde{f}(x_0, t_0)<0,
		\]
		whence	
		\[
		0\le \frac{\partial \tilde{f}}{\partial t}(x_0, t_0)< \Delta \tilde{f}(x_0, t_0).
		\]
		Then there exists a neighborhood $\mathcal{U}\subset \bar \Omega \times \{ t_0\}$ of $(x_0, t_0)$ such that 
		\[ \Delta \tilde{f}> 0 \,\,\, {\rm in} \,\, \mathcal{U}.
		\] Since the domain $\bar \Omega $ is convex with smooth boundary, and therefore satisfies the interior sphere condition at $x_0$, we can apply the well-known elliptic Hopf Lemma [Lemma 3.4, \ref{GT:01}] in order to achieve $\langle \nabla \tilde{f}(x_0, t_0), \mu \rangle >0$, which is a contradiction with the hypothesis. Therefore 
		\[
		\sup _{\bar \Omega \times [0, T)}f\leq e^ {\epsilon t}\max \{\max _{\bar \Omega }f(\cdot, 0) ,0 \}.
		\]
		Sending $\epsilon \to 0 $ we have
		\[
		\sup _{\bar \Omega \times [0, T)}f\leq \max \{\max _{\bar \Omega }f(\cdot, 0) ,0 \}.
		\]
		Setting $S=\max _{\bar \Omega }f(\cdot, 0)$ and applying the result we have just proved to $f+S+\epsilon $, with $\epsilon >0$ the above equation provides
		\[
		\sup _{\bar \Omega \times [0, T)}f\leq \max _{\bar \Omega }f(\cdot, 0) .
		\]
		The second claim of the lemma follows by applying the first one to $-f$.
	\end{proof}	
	
	Henceforth we will refer to the lemma above just as 
	{\it maximum principle}. The following lemma is an adaptation of [\ref{CK:04}, Theorem 4.4]
	
	\begin{lemma}\label{smp} 
		Let $f:\bar \Omega  \times [0, T)\to \mathbb{R}$ be a function twoce differentiable in space and once in time. Let $Q=\Big( \frac{\partial }{\partial t}-\Delta \Big) $, where $\Delta $ is the time-dependent Laplace-Beltrami operator. Suppose that $f$ satisfies 
		$$\begin{cases} 
		Q f \le  \langle Y, \nabla f\rangle + F(f), \, \mbox{ for all } \, (x, t)\in \bar \Omega \times [0, T) & \\ 
		\langle \nabla f, \mu \rangle \geq 0 \,\mbox{ for all } \, (x, t)\in \partial \bar \Omega \times [0,T), & 
		\end{cases} $$
		where $Y$ is a time-dependent vector field in $\bar \Omega  \times [0,T)$  and $F:\mathbb{R}\to \mathbb{R}$ is a locally Lipschitz function. Suppose there exists $c\in \mathbb{R}$ such that $f(x, 0)\leq c$ for all $x\in \bar \Omega $, and let $\gamma  $ be a solution to the associated ordinary differential equation 
		$$ \frac{d\gamma }{dt}=F(\gamma )$$ 
		satisfying $\gamma (0)=c.$  Then 
		\[
		f(x, t)\leq \gamma (t)
		\]
		for all $x\in \bar\Omega $ and $t\in [0, T)$ such that $\gamma  (t)$ exists. 
	\end{lemma} 
	
	\begin{proof}
		Let $\tilde{t}\in [0, T).$ Since $\bar \Omega $ is compact, then there exists a constant $C_{\tilde{t}}<\infty $ such that 
		\begin{center}
			$|f(x, t)|\le C_{\tilde{t}}$ and $|\gamma (t)|\le C_{\tilde{t}}$  $\forall (x, t)\in \bar \Omega \times [0, \tilde{t}].$
		\end{center}
		Since $F$ is locally Lipschitz, there exists a constant $L_{\tilde{t}}<\infty $ such that 
		\begin{center}
			$|F(a)-F(b)|\le L_{\tilde{t}}|a-b|$   $\forall a, b\in [-C_{\tilde{t}}, C_{\tilde{t}}].$
		\end{center}
		Notice that
		\begin{eqnarray}\label{heat-gamma-f}
		Q(\gamma -f)=-Qf+\dfrac{d\gamma }{dt}&\ge &\langle Y, \nabla f\rangle +F(\gamma )-F(f)\nonumber \\
		&\ge &\langle Y, \nabla (\gamma -f)\rangle -sign (\gamma -f)L_{\tilde{t}} (\gamma - f),
		\end{eqnarray}
		where $sign (\gamma -f)$ is $-1, 0, 1$ depending whether $\gamma -f$ is negative, zero or positive, respectively.
		
		Now we set $D_{\tilde{t}}=-L_{\tilde{t}}sign (\gamma - f)$ and define 
		\[
		J(x, t):=e^ {-D_{\tilde{t}}t}(\gamma -f).
		\]
		We compute using \eqref{heat-gamma-f}
		\[
		QJ =  e^ {-D_{\tilde{t}}t}Q(\gamma -f )-D_{\tilde{t}}J \\
		\ge e^ {-D_{\tilde{t}}t} \langle Y, \nabla (\gamma - f)\rangle =\langle Y, \nabla J\rangle .
		\]
		We can thus apply Lemma \ref{wmp} to $J$ and achieve, since $\gamma -f\ge 0$ at $t=0$ from hypothesis,
		\[
		J(x, t)\ge \min _{\bar \Omega }J(x, 0)\geq 0 \,\,\, \forall (x, t)\in \bar \Omega \times [0, \tilde{t}],
		\]
		what means that 
		\[
		f(x, t)\le \gamma (t) \,\,\, \forall (x, t)\in \bar \Omega \times [0, \tilde{t}].
		\]
		Since $\tilde{t}$ was arbitrary, we finish the proof.
	\end{proof}
	
	\subsection{\emph{A priori} $C^0$ bounds}  \label{gradient} 
	
	
	The next two results provide $C^0$ \emph{a priori} estimates to \eqref{F2}.
	
	\begin{lemma}
		Suppose that \eqref{max-princ} and \eqref{HM} hold. Then
		\begin{equation}
		\label{height-1}
		\min_{\Sigma_0} u + C_+ nt \le u (x, t)\le \max_{\Sigma_0} u + C_- nt  \,\,\, {\rm for\,\, all} \,\,\ (x, t)\in \bar\Omega \times [0,T].
		\end{equation}
		This estimate can be refined as
		\begin{equation}
		\label{height-2}
		\frac{C_+}{C_-} \min_{\Sigma_0} u +  C_+ n t\le u (x, t) \le  \frac{C_-}{C_+}\max_{\Sigma_0} u +  C_- n t  \,\,\, {\rm for\,\, all} \,\,\ (x, t)\in \bar\Omega \times [0,T].
		\end{equation}
	\end{lemma}
	
	\begin{proof} We have from \eqref{heat-s}  that
		\[
		Qs = n \frac{\rho '}{\rho }- n \frac{\rho ' }{\rho} \Big\langle \nabla s, \frac{\partial}{\partial s} \Big\rangle  \le C_- n  - n\frac{\rho ' }{\rho} \Big\langle \nabla s, \frac{\partial}{\partial s} \Big\rangle.
		\]
		Therefore
		\[
		Qs +n\frac{\rho ' }{\rho}\Big\langle \nabla s, \frac{\partial}{\partial s} \Big\rangle \le C_- n.
		\]
		In the same way we have
		\[
		Qs +n\frac{\rho ' }{\rho}\Big\langle \nabla s, \frac{\partial}{\partial s} \Big\rangle \ge C_+ n.
		\]
		Since $\langle \nabla s, \mu\rangle =0$ along $\partial\Sigma_t$ we conclude from the maximum principle that 
		\[
		\min_{\Sigma_0} u + C_+ nt \le u \le \max_{\Sigma_0} u + C_- nt
		\]
		Now, we have from \eqref{heat-kappa}   that
		\[
		Q\kappa = C_0n + C_0 \frac{\rho ''}{\rho} \frac{\rho}{\rho ' } |\nabla s|^2 = C_0n + C_0 \frac{\rho ''}{\rho} \langle \nabla \kappa, \nabla s\rangle
		\]
		Therefore
		\[
		Q\kappa  + C_0 \frac{\rho ''}{\rho} \Big\langle \nabla \kappa, \frac{\partial}{\partial s}\Big\rangle =C_0 n.
		\]
		Using \eqref{nablakappa} we conclude from the maximum principle that 
		\begin{equation*}
			\min_{\Sigma_0} \kappa + C_0 n t\le \kappa  \le \max_{\Sigma_0} \kappa + C_0 n t
		\end{equation*}
		Since
		\[
		\frac{C_0}{C_-}s \le \kappa(s) \le \frac{C_0}{C_+} s
		\]
		we conclude that
		\begin{equation*}
			\frac{C_+}{C_0} \min_{\Sigma_0} \kappa +  C_+ n t\le u  \le \frac{C_-}{C_0}\max_{\Sigma_0} \kappa +  C_- n t .
		\end{equation*}
		Therefore,
		\begin{equation*}
			\frac{C_+}{C_-} \min_{\Sigma_0} u +  C_+ n t\le u  \le \frac{C_-}{C_+}\max_{\Sigma_0} u +  C_- n t .
		\end{equation*}
		
		%
		%
		%
		This finishes the proof. 
	\end{proof}

	\begin{prop}
		Suppose that \eqref{max-princ} and \eqref{HM} hold. Then
		\begin{equation}
		\label{est-rho}
		e^{C_+^2 n t} \min_{\Sigma_0} \rho \le \rho (x, t)\le  e^{C_-^2 n t}\max_{\Sigma_0} \rho, \,\,\, {\rm for all} \,\,\ (x, t)\in \bar\Omega \times [0,T].
		\end{equation}
	\end{prop}
	
	\begin{proof} We have from \eqref{heat-s}
		\begin{align*}
			Q \rho = \rho' Qs - \rho '' |\nabla s|^2 =n \frac{\rho'^2}{\rho} +\rho\bigg(\frac{\rho'^2}{\rho^2}  -\frac{ \rho '' }{\rho} \bigg)|\nabla s|^2
		\end{align*}
		from what follows that
		\begin{align*}
			Q\rho + \frac{\rho}{\rho'}\bigg(\frac{\rho'^2}{\rho^2}  -\frac{ \rho '' }{\rho} \bigg)\Big\langle \nabla \rho, \frac{\partial}{\partial s}\Big\rangle = n \frac{\rho'^2}{\rho^2}\rho. 
		\end{align*}
		Hence, \eqref{HM} implies that
		\begin{align*}
			C_+^2 n \rho \le Q\rho + \frac{\rho}{\rho'}\bigg(\frac{\rho'^2}{\rho^2}  -\frac{ \rho '' }{\rho} \bigg)\Big\langle \nabla \rho, \frac{\partial}{\partial s}\Big\rangle\le   C_-^2n\rho. 
		\end{align*}
		Using \eqref{nablarho} we conclude from the maximum principle that 
		\[
		e^{-C_-^2 n t}\rho\le \max_{\Sigma_0} \rho.
		\]
		and
		\[
		\min_{\Sigma_0} \rho \le e^{-C_+^2 n t}\rho
		\]
		Therefore
		\begin{equation}
		e^{C_+^2 n t} \min_{\Sigma_0} \rho \le \rho\le  e^{C_-^2 n t}\max_{\Sigma_0} \rho. 
		\end{equation}
		This finishes the proof. 
	\end{proof}

	\subsection{A priori $C^1$ estimates} In this section we exhibit two variants of \emph{a priori} gradient estimates for solutions of \eqref{F2}. Before presenting the results we make the below remark that will be useful along the performing of our estimates.
	
	\begin{remark}\label{tcc}
		Given a vector field $Y$ on $N$,  we have the following relation between the Ricci tensor $\operatorname{Ric}$ of the leaf $M$ and the Ricci tensor 
		$\overline{\Ric}$ of the ambient:
		\begin{align}
			\label{riccii}
			\overline{\Ric}(Y, Y) =  n\frac{\rho ''}{\rho }\langle Y, Y\rangle+ \Ric(Y^M , Y^M) - (n-1)\Big(\frac{\rho ' }{\rho}\Big) ' \langle Y^M, Y^M\rangle,
		\end{align}
		where $Y^M = Y + \langle Y, \frac{\partial}{\partial s}\rangle \frac{\partial}{\partial s}$ and
		\[
		\langle Y^M, Y^M\rangle =\rho^2 \sigma(Y^M, Y^M).
		\]
		We refer the reader to [\ref{O'N:83}, Corollary 7.43]  for further details. In particular, applying the equation \eqref{riccii} to any light-like vector field one can easily check that the validity of the null convergence condition is equivalent to have the following inequation verified
		\begin{equation*}
			\Ric - (n-1) \rho^2\Big(\frac{\rho ' }{\rho}\Big) ' \sigma\ge 0.
		\end{equation*}
		Therefore a consequence of the validity of the null convergence condition is
		\begin{equation}
		\label{TCC}
		\overline{\operatorname{Ric}}(\nu , \nu )+n\frac{\rho ''}{\rho }\ge 0.
		\end{equation}
	\end{remark}
	
	\begin{lemma}
		Suppose that \eqref{max-princ} and \eqref{HM} hold and that $N$ obeys the null convergence condition. Then 
		\begin{equation}
		\label{est-Du-0}
		|Du|(x,t)\le C(n, \max_{\Sigma_0} |H|,  \max_{\Sigma_0}\Theta, \max_{\Sigma_0} \phi, \min_{\Sigma_0} \rho, C_-, T)\rho,
		\end{equation}
		for all $(x,t) \in \bar\Omega \times [0, T]$, where the constant in the right-hand side is always strictly less than $1$.
	\end{lemma}
	
	\begin{proof}
		%
		Note that \eqref{max-princ},
		\eqref{heat-H}, \eqref{TCC} imply that
		\begin{align*}
			QH^2 = -2H^2 (|A|^2+ \overline{\rm Ric} (\nu, \nu)) - 2|\nabla H|^2\le  2n\frac{\rho '' }{\rho}H^2 \le 2n\frac{\rho'^2 }{\rho^2}H^2  \le 2nC_-^ 2 H^2 
		\end{align*}
		for all $t\in [0,T]$. Hence from the maximum principle
		\begin{equation*}
			\label{sub-H2}
			Q(e^{-2nC_-^ 2t} H^2) \le 0.
		\end{equation*}
		Moreover one has from \eqref{nablaH2} and Remark \ref{rem-conv} that
		\[ 
		\langle \nabla H^2, \mu\rangle = -2H^2 II^K (\nu, \nu)\le 0,
		\]
		what implies again for the maximum principle that 
		\begin{equation*}
			H^2\leq  \max_{\Sigma_0} H^2 \cdot e^ {2nC_-^ 2T}=:C_1^ 2 \,\,\, {\rm in } \,\, \bar\Omega \times [0,T].
		\end{equation*}
		Then we have
		\begin{equation}
		\label{bound-H}
		|H| \le C_1 \,\, \mbox{ in } \,\, \bar \Omega  \times [0,T].
		\end{equation}
		Set $0<E \le \frac{n}{2C_1}$. Hence, we have from \eqref{heat-phi}, \eqref{heat-Theta} and \eqref{TCC} that
		\begin{eqnarray*}
			Q (E\Theta - \phi) &= &2E\rho ' H- E\Theta \left(|A|^2+{\rm \overline{Ric}}(\nu,\nu)+n\frac{\rho ''}{\rho }\right) -n\rho'\\
			& \le &(2C_1E-n )\rho' \le 0 .
		\end{eqnarray*}
		From \eqref{nablarho} and \eqref{nablatheta} we get
		\begin{equation*}
			\langle \nabla (E\Theta-\phi), \mu\rangle = -E\Theta II^K (\nu, \nu)\le 0.
		\end{equation*}
		Therefore the maximum principle and \eqref{max-princ} implies that 
		\begin{eqnarray}
		\label{gradientestimate} 
		E\Theta  \le \max_{\Sigma_0} (E\Theta-\phi)+\max _{\bar \Omega \times [0, T]}\phi =:M \,\,\, {\rm in}\,\,\, \bar \Omega \times [0, T].
		\end{eqnarray} 
		It follows from \eqref{thetaW} that
		\[
		\frac{\rho^2}{W} \le \frac{M}{E}.
		\]
		Hence
		\[
		|Du|^2 \le \rho^2 \Big(1-\frac{E^2\rho^2}{M^2}\Big).
		\]
		On the other hand, \eqref{thetarho} implies that
		\[
		E\rho\leq E\Theta \leq M .
		\] 
		One concludes from \eqref{est-rho} that 
		\begin{equation}
		\label{est-W}
		|Du| \le \rho \Big( 1- \frac{E^2}{M^2}e^{2C_+^2 n t} \min_{\Sigma_0} \rho^2\Big)^{\frac{1}{2}}.
		\end{equation}
		In this way we obtained a (time-dependent) gradient estimate of the form
		\begin{equation}
		\label{est-Du}
		|Du| \le C(n, \max_{\Sigma_0} |H|,  \max_{\Sigma_0}\Theta, \max_{\Sigma_0} \phi, \min_{\Sigma_0} \rho, C_+, T)\rho.
		\end{equation}
		where the constant in the right-hand side is always strictly less than $1$. This finishes the proof.
	\end{proof}
	
	A variant of this gradient estimate is written in the following proposition.

	\begin{prop}
		\label{est-grad-2}
		Suppose that \eqref{max-princ} and  \eqref{HM} and that $N$ obeys the null convergence condition. Then
		\begin{equation}
		\label{est-conv}
		|Du|(x, t) \le \rho \Big(1-\frac{1}{\mathcal{M}^2}\Big)^{\frac{1}{2}}
		\end{equation}
		for all $(x,t) \in \bar\Omega \times [0, T]$, where $\mathcal{M}=\mathcal{M}(T)$ is a constant bigger than $1$ for $T>0$.
	\end{prop}
	
	\begin{proof} Note that 
		\begin{align*}
			Q \rho = \rho' Qs - \rho '' |\nabla s|^2 =n \frac{\rho'^2}{\rho} +\rho\bigg(\frac{\rho'^2}{\rho^2}  -\frac{ \rho '' }{\rho} \bigg)|\nabla s|^2
		\end{align*}
		Denoting
		\[
		\theta = \frac{\Theta}{\rho}
		\]
		one has taking \eqref{heat-Theta} into account
		\begin{align*}
			Q\theta - \frac{2}{\rho}\langle \nabla \theta, \nabla\rho\rangle  & = \frac{1}{\rho} Q\Theta -\frac{\Theta}{\rho^2} Q\rho \\
			& = 2\frac{\rho' }{\rho} H-\theta \left( {\rm \overline{Ric}}(\nu,\nu)+|A|^2+n\frac{\rho''}{\rho}\right) - \theta \bigg(n \frac{\rho'^2}{\rho^2} +\bigg(\frac{\rho'^2}{\rho^2}  -\frac{ \rho '' }{\rho} \bigg)|\nabla s|^2\bigg).
		\end{align*}
		Hence \eqref{TCC} implies that
		\begin{align*}
			Q\theta - \frac{2}{\rho}\langle \nabla \theta, \nabla\rho\rangle \le 2\frac{\rho' }{\rho} H- \theta \bigg(n \frac{\rho'^2}{\rho^2} +\bigg(\frac{\rho'^2}{\rho^2}  -\frac{ \rho '' }{\rho} \bigg)|\nabla s|^2\bigg).
		\end{align*}
		Using  \eqref{max-princ}, \eqref{HM}, \eqref{thetarho} and \eqref{bound-H} we achieve
		\begin{equation*}
			Q\theta - \frac{2}{\rho}\langle \nabla \theta, \nabla\rho\rangle \le 2\frac{\rho' }{\rho} H \le 2C_- C_1.
		\end{equation*}
		Taking in account expressions \eqref{nablarho} and \eqref{nablatheta} one deduces from the maximum principle that
		\begin{align}
			\label{matM}
			\theta \leq 2C_-C_1T+\max _{\Sigma _0}\theta =:\mathcal{M} \,\,\, {\rm in} \,\, \bar\Omega \times [0,T].
		\end{align}
		We remark that from \eqref{thetarho} we get that $\mathcal{M}\ge 1$ and $\mathcal{M}>1$ for $T>0$. Taking equation \eqref{thetaW} into account we accomplish \eqref{est-conv}.
	\end{proof}

	%
	
	\subsection{Mean convexity}\label{mean}  The next result establishes that the evolving graphs $\Sigma_t$ remain mean convex if the initial graph $\Sigma_0$ is supposed to be mean convex, under the assumption that $\rho$ satisfies \eqref{curv-hip}. This is the case when $N$ has constant sectional curvature $\gamma $ for instance.
	
	\begin{prop}
		\label{mean-convex}
		Suppose that $\Sigma _ 0$ is mean convex ($H(\cdot , 0) >0$) and
		\begin{equation}
		\label{curv-hip}
		\sup _{\bar \Omega \times [0, T]}-\frac{\rho '' (s)}{\rho(s)} = \lambda < +\infty
		\end{equation}
		for some constant $\lambda  $. Then
		\begin{equation*}
			\label{est-conv-H}
			H (\cdot , t)>0\,\,\,\, {\rm for all} \,\, t\in [0, T).
		\end{equation*}
	\end{prop}
	
	%
	
	\begin{proof} Denoting
		\[
		v = \frac{H}{\Theta}
		\]
		we get
		\begin{align*}
			Qv =  \frac{1}{\Theta} QH - \frac{H}{\Theta^2}  Q\Theta + \frac{2}{\Theta^2} \langle \nabla H, \nabla\Theta\rangle -\frac{2|\nabla\Theta|^2}{\Theta^3}  H.
		\end{align*}
		However
		\[
		\frac{2}{\Theta^2} \langle \nabla H, \nabla\Theta\rangle = \frac{2}{\Theta}\langle \nabla v, \nabla\Theta\rangle + \frac{2|\nabla\Theta|^2}{\Theta^3}H,
		\]
		then we conclude that
		\begin{equation}\nonumber 
		\label{Qv}
		Qv - \frac{2}{\Theta}\langle \nabla v, \nabla \Theta\rangle =  \frac{1}{\Theta} QH - \frac{H}{\Theta^2}  Q\Theta.
		\end{equation}
		Using \eqref{heat-Theta} and \eqref{heat-H} one obtains
		\begin{align*}
			Qv - \frac{2}{\Theta}\langle \nabla v, \nabla \Theta\rangle & = -(|A|^2+ {\rm \overline{Ric}}(\nu, \nu )) v - 2\rho'(s) v^2 + \Big( {\rm \overline{Ric}}(\nu,\nu)+|A|^2+n\frac{\rho''(s)}{\rho(s)}\Big)v\\
			& = - 2\rho'(s) v^2 + n\frac{\rho''(s)}{\rho(s)}v .
		\end{align*}
		Consequently
		\begin{eqnarray*}
			\frac{1}{2}Qv^ 2 & =&vQv-|\nabla v|^ 2 \\
			&= &2\frac{v}{\Theta }\langle \nabla v, \nabla \Theta \rangle -\frac{1}{2v}\langle \nabla v^ 2, \nabla v\rangle -2\rho '(s)v^ 3 +n\frac{\rho ''(s)}{\rho (s)}v^ 2.
		\end{eqnarray*}
		Defining  $\eta(t)= e^{2n\lambda t}$, one concludes from \eqref{curv-hip} that
		\[
		2n\frac{\rho '' (s)}{\rho(s) } \eta(t) + \eta ' (t) \ge 0.
		\] 
		Thus 
		\begin{eqnarray}\label{heat-eta-v}
		Q(\eta v^ 2)+ \Big \langle 2\nabla (\eta v^ 2), \frac{\nabla v}{2v} - \frac{\nabla \Theta }{\Theta }\Big \rangle  & = &  -4\rho '(s) \eta v^ 3 +2n\frac{\rho ''(s)}{\rho (s)} \eta v^ 2 +v^ 2 \eta '\\
		& \ge & -4\rho '(s) \eta v^ 3. \nonumber 
		\end{eqnarray}
		Remark that from \eqref{thetarho} and \eqref{bound-H} 
		\begin{equation*}
			-4\rho '(s)v \ge -4\rho '(s)\frac{|H|}{\Theta }\ge -4C_1\frac{\rho '(s)}{\rho (s)}\frac{\rho (s)}{\Theta }\ge -4C_-C_1=:-\beta  .
		\end{equation*}
		Coming back to \eqref{heat-eta-v} we obtain
		\[ 
		Q(\eta v^ 2)+ \Big \langle 2\nabla (\eta v^ 2), \frac{\nabla v}{2v} - \frac{\nabla \Theta }{\Theta }\Big \rangle \ge -\beta \eta v^ 2.
		\]
		We also note that \eqref{nablatheta} and \eqref{nablaH2} yield
		\[
		\langle \nabla v, \mu\rangle = \frac{1}{\Theta}\langle \nabla H, \mu\rangle -\frac{H}{\Theta^2} \langle \nabla \Theta, \mu\rangle = -\frac{H}{\Theta} II^K(\nu, \nu) +\frac{H}{\Theta} II^K(\nu, \nu)=0
		\]
		along $\partial\Sigma_t \subset K$.  
		Therefore from the maximum principle we get
		\[
		e^ {\beta t}\eta v^ 2\ge \min _{\Sigma _0}v^ 2>0,
		\]
		where the last inequality holds since $\Sigma _0 $ is mean convex. Thus we conclude that $H^2$ never attains zero, therefore $H>0$ along the whole flow.
	\end{proof}
	
	
	\subsection{Second order \emph{a priori} estimates.} Now we obtain second order estimates. For doing that, we need to use a cut-off function. Let $r$ is the geodesic distance $\operatorname{dist}(o,\cdot)$ from a given point $o\in \bar \Omega $ and  $R>0$ is chosen such that the geodesic ball $B_R(o)$ centered at $o\in \bar \Omega$ is contained in $\bar \Omega$.
	
	Let $\chi $ be the solution of the Jacobi equation 
	$$\chi ''(r)+K_{\bar\Omega }\chi (r)=0,$$
	with $\chi (0)=0$ and $\chi '(0)=1$, where $K_{\bar\Omega }$ is the sectional curvature of $\bar \Omega$. 
	Given a point $(x, s)$ in the region bounded by $K$ in $N$ (that is, the set of flow lines of $X$ that cross $\bar\Omega$) one sets 
	\[
	\xi(x,s)=\xi(x)= \xi(r)
	\]
	if $r = \operatorname{dist}(o,x)$ with a slight abuse of notation.  A possible choice of a compactly supported radial function is $\xi(r) = \varrho^3(r)$, where 
	\[
	\varrho(r) = (C_R - \chi(r))_+
	\]
	and $C_R \ge 2\chi(R)$.
	
	\begin{lemma} 
		Let $\xi$ as defined above. Hence
		\begin{equation*}
			Q\xi|_{\Sigma_t} = -\bar\Delta \xi - \langle \bar\nabla_\nu \bar\nabla \xi, \nu\rangle .
		\end{equation*}
		In addition
		\begin{equation*}
			Q\xi|_{\Sigma_t} \le   \mathcal{N}\big(2C_- ||D\xi||_{M_s} + ||DD\xi||_{M_s}\big) 
		\end{equation*}
		in $\bar \Omega \times [0, T]$ for some constant $\mathcal{N}=\mathcal{N}(T)$.
	\end{lemma}
	
	\begin{proof} 
		In what follows, we also denote by $\xi$ the restriction of this function to a graph $\Sigma_t$.  Fixed this notation, one has
		\begin{align*}
			\nabla \xi = \bar\nabla \xi + \langle \bar\nabla \xi, \nu\rangle \nu.
		\end{align*}
		Hence using an orthonormal local frame $\{{\sf e}_k\}_{k=1}^{n}$ in $\Sigma _t$
		\begin{align*}
			\Delta \xi|_{\Sigma_t} &= \sum_k \langle \bar\nabla_{{\sf e}_k} \nabla \xi, {\sf e}_k\rangle =  \sum_k \langle \bar\nabla_{{\sf e}_k} \bar\nabla \xi, {\sf e}_k\rangle +\sum_k \langle \bar\nabla_{{\sf e}_k} \nu, {\sf e}_k\rangle\langle \bar\nabla \xi, \nu\rangle\\
			&  = \bar\Delta \xi + \langle \bar\nabla_\nu \bar\nabla \xi, \nu\rangle + H \langle \bar\nabla \xi, \nu\rangle.
		\end{align*}
		On the other hand
		\[
		\frac{\partial \xi}{\partial t} = \langle \bar\nabla \xi, H\nu\rangle = H \langle\bar\nabla \xi, \nu\rangle,
		\]
		whence we conclude that
		\begin{equation}
		\label{heat-chi}
		Q\xi|_{\Sigma_t} = -\bar\Delta \xi - \langle \bar\nabla_\nu \bar\nabla \xi, \nu\rangle .
		\end{equation}
		Now we compute
		\begin{align*}
			& \langle \bar\nabla_\nu \bar\nabla\xi, \nu\rangle = \Big\langle \nu, \frac{\partial}{\partial s}\Big\rangle^2 \Big\langle \bar\nabla_{\frac{\partial}{\partial s}} \bar\nabla\xi, \frac{\partial}{\partial s}\Big\rangle  - 2\nu^i \Big\langle \nu, \frac{\partial}{\partial s}\Big\rangle \Big\langle \bar\nabla_{\frac{\partial}{\partial x^i}} \bar\nabla\xi, \frac{\partial}{\partial s}\Big\rangle +\nu^i \nu^j \Big\langle \bar\nabla_{\frac{\partial}{\partial x^i}} \bar\nabla\xi, \frac{\partial}{\partial x^j}\Big\rangle
		\end{align*}
		where $\nu^j = g^{jk} \langle\nu, \frac{\partial}{\partial x^k}\rangle$. Therefore since $\langle \bar\nabla \xi, \frac{\partial}{\partial s}\rangle =0$ one has using \eqref{nabla-frame}
		\begin{align*}
			\langle \bar\nabla_\nu \bar\nabla\xi, \nu\rangle &=    2\nu^i \Big\langle \nu, \frac{\partial}{\partial s}\Big\rangle \Big\langle  \bar\nabla\xi, \bar\nabla_{\frac{\partial}{\partial x^i}}\frac{\partial}{\partial s}\Big\rangle +\nu^i \nu^j \Big\langle \bar\nabla_{\frac{\partial}{\partial x^i}} \bar\nabla\xi, \frac{\partial}{\partial x^j}\Big\rangle\\
			& \,\, = 2 \frac{\rho' (s) }{\rho(s) }\Big\langle \nu, \frac{\partial}{\partial s}\Big\rangle \nu^i \xi_i +\nu^i \nu^j \Big\langle \bar\nabla_{\frac{\partial}{\partial x^i}} \bar\nabla\xi, \frac{\partial}{\partial x^j}\Big\rangle\\
			& \,\, = 2 \frac{\rho' (s) }{\rho(s) }\Big\langle \nu, \frac{\partial}{\partial s}\Big\rangle \langle D\xi, \nu^{\sf t}\rangle +\nu^i \nu^j \Big\langle D_{\frac{\partial}{\partial x^i}} D\xi, \frac{\partial}{\partial x^j}\Big\rangle,
		\end{align*}
		where $D$ is the induced connection in $M_s:=\{s\} \times M  \subset N$ and $\nu^{\sf t} = \nu +\langle \nu, \frac{\partial}{\partial s}\rangle \frac{\partial}{\partial s}$. Since $\bar\nabla \xi|_{M_s} = D\xi$ we also have
		\begin{align*}
			& \bar\Delta\xi =    g^{ij} \Big\langle \bar\nabla_{\frac{\partial}{\partial x^i}} \bar\nabla\xi, \frac{\partial}{\partial x^j}\Big\rangle - \Big\langle \bar\nabla_{\frac{\partial}{\partial s}} \bar\nabla\xi, \frac{\partial}{\partial s}\Big\rangle = 
			g^{ij} \Big\langle D_{\frac{\partial}{\partial x^i}} D\xi, \frac{\partial}{\partial x^j}\Big\rangle = \Delta_{M_s} \xi|_{M_s},
		\end{align*}
		where $\Delta_{M_s}$ is the Laplace-Beltrami operator in $M_s$ with respect to the induced metric $g_{ij}=\rho(s)^2 \sigma_{ij}$. 
		We  conclude that
		\begin{align*}
			Q\xi|_{\Sigma_t}& = -\Delta_{M_s} \xi|_{M_s} -2 \frac{\rho' (s) }{\rho(s) }\Big\langle \nu, \frac{\partial}{\partial s}\Big\rangle \langle D\xi, \nu^{\sf t}\rangle -\nu^i \nu^j \Big\langle D_{\frac{\partial}{\partial x^i}} D\xi, \frac{\partial}{\partial x^j}\Big\rangle\\
			& \,\, \le 2 \frac{\rho' (s) }{\rho(s) }  ||D\xi||_{M_s} \frac{\Theta}{\rho}\Big(\Big(\frac{\Theta}{\rho}\Big)^2-1\Big)^{\frac{1}{2}} + ||DD \xi||_{M_s}\Big(\frac{\Theta}{\rho}\Big)^2  \\
			& \,\, \le \mathcal{N}\big(2C_- ||D\xi||_{M_s} + ||DD\xi||_{M_s}\big) ,
		\end{align*}
		where $\mathcal{N}=\max \{\mathcal{M}(\mathcal{M}^2-1)^{1/2}, \,\, \mathcal{M}^ 2 \}$, being $\mathcal{M}$ defined in \eqref{matM}. This finishes the proof.
	\end{proof}

	\begin{prop}
		\label{simons}
		We have
		\begin{eqnarray*}
			\frac{1}{2} Q|A|^2 +|\nabla A|^2  & = & |A|^4+T_{ij} a^{ij} +
			g^{k\ell}(\bar R^s_{ki\ell}a_{sj}a^{ij}+\bar R^s_{kij}a_{\ell s} a^{ij}) -H\langle \overline{ R}(\partial _ i, \nu )\nu , \partial _ j\rangle ,
		\end{eqnarray*}
		where $T_{ij}=g^{k\ell}(\nabla_i L_{k\ell j}+\nabla_k L_{\ell ij}) a^{ij}.$
		
	\end{prop}

	\begin{proof} See appendix. 
	\end{proof}

	\begin{prop}
		\label{hess-est-prop}
		Suppose that for all $x\in M, {\sf v} \in T_x M$ with $\langle Dr, {\sf v}\rangle =0$ and
		\begin{equation}
		\label{comp-thm}
		K_M (Dr \wedge {\sf v})\ge -\chi (r)
		\end{equation}
		where $K_M (Dr \wedge {\sf v})$ denotes the sectional curvature of a plane containing $Dr$. Then
		\begin{equation}
		\label{hess-est}
		|A|^2 < \max\Big \{|\mathcal{L}|^{\frac{1}{3}}, |\bar C|^{\frac{1}{2}}, \frac{C_4-T}{C_5(C_6T+C_7)} \Big \},
		\end{equation}
		in $\bar \Omega \times [0, T]$, where 
		\[
|\mathcal{L}|=\sup _{\bar \Omega \times [0,T]}\Big ( |T_{ij}+\frac{1}{\sqrt{n}}\langle \bar R(\partial _ i, \nu ) \nu ,\partial _ j\rangle | \Big ) 
		\]
		and
		\[
		 |\bar C|=\cdots 
		\]
		$C_4 $ and $C_5$ positive constants and $C_7$ a negative constant.
	\end{prop}

	\begin{proof} Let $f(t) = C_4+C_5 t$ for some positive constants $C_4, C_5$ to be fixed later. Hence, denoting $h = f|A|^2 \xi$ one has
		\begin{align*}
			&  Qh + 2\frac{1}{\xi}\langle \nabla \xi, \nabla h\rangle = f|A|^2 Q\xi + \xi Q(f|A|^2)  + 2 h\frac{|\nabla \xi|^2}{\xi^2} \\
			& \,\, \le  f|A|^2 Q\xi + \xi f'(t) |A|^2 +2 \xi f(t) (- |\nabla A|^2+|A|^4+|\mathcal{L}| |A| + |\bar C||A|^2) + 2h\frac{|\nabla \xi|^2}{\xi^2},
		\end{align*}
	with $|\mathcal{L}|=\sup _{\bar \Omega \times [0,T]}\Big ( |T_{ij}+\frac{1}{\sqrt{n}}\langle \bar R(\partial _ i, \nu ) \nu ,\partial _ j\rangle | \Big )$ and $|\bar C|=\cdots $.

		Note that 
		\[
		|\nabla \xi|^2 = ||D\xi||_{M_s}^2 + \langle\bar\nabla\xi, \nu\rangle^2 \le ||D\xi||_{M_s}^2 + ||D\xi||_{M_s}^2 \Big(\Big(\frac{\Theta}{\rho}\Big)^2-1\Big) \le ||D\xi||_{M_s}^2 \mathcal{M}^2.
		\]
		Suppose that
		\begin{equation}\label{assumption}
		|A| \ge \max\{|\mathcal{L}|^{\frac{1}{3}}, |\bar C| ^{\frac{1}{2}}\}.
		\end{equation}
		Otherwise we are done. Hence 
		\begin{align}\label{heat-h} 
			Qh + 2\frac{1}{\xi}\langle \nabla \xi, \nabla h\rangle & \le  \frac{1}{\xi}\Big(2C_- ||D\xi||_{M_s} + 2\frac{1}{\xi} ||D\xi||_{M_s}^2+ ||DD\xi||_{M_s}\Big) \mathcal{N} h \\
			&\,\,\,\,\,\,\, + \frac{f'(t)}{f(t)} h 
			+ \frac{6}{f\xi} h^2. \nonumber 
		\end{align}
		Hence, $D\xi = \xi'(r) Dr$ and $DD\xi = \xi'(r) DDr + \xi '' (r) Dr\otimes Dr$. We have
		\[
		\xi'(r) = -3\varrho ^2(r) \chi'(r) \quad \mbox{ and } \quad \xi '' (r) = 6\varrho(r) \chi'^2(r) -3\varrho^2(r) \chi '' (r).
		\]
		In this case \eqref{comp-thm} allow us to call the Hessian comparison theorem [Theorem 1.4, \ref{AMR:16}] to achieve
		\begin{eqnarray*}
			DDr\leq \frac{\chi '(r)}{\chi (r)}\Big ( \langle \, , \, \rangle -{\rm d}r\otimes  {\rm d}r \Big ).
		\end{eqnarray*}
		Consequently
		\begin{align*}
			\frac{1}{\xi}\Big\langle D_{\frac{\partial}{\partial x^i}} D\xi, \frac{\partial}{\partial x^j} \Big\rangle {\sf v}^i {\sf v}^j  \le -3\frac{\varrho^2(r)}{\varrho^3(r)} \frac{\chi'^2(r)}{\chi(r)} (\langle {\sf v}, {\sf v}\rangle-\langle Dr, {\sf v}\rangle^2)  +\bigg( 6 \frac{\chi'^2(r)}{\varrho^2(r)} -3\frac{\chi'' (r)}{\varrho(r)}\bigg)  \langle Dr, {\sf v}\rangle^2
		\end{align*}
		for any tangent vector ${\sf v}$ in $M_s$.
		Therefore  $C_R= 2 \chi(R)\ge 2\chi (r)$ in $\operatorname{supp} \xi$ implies that
		\begin{align*}
			\frac{1}{\xi}\Big\langle D_{\frac{\partial}{\partial x^i}} D\xi, \frac{\partial}{\partial x^j} \Big\rangle {\sf v}^i {\sf v}^j & \le 3 \frac{\chi'^2(r)}{C_R^2} \langle Dr, {\sf v}\rangle^2 +\bigg( 6 \frac{\chi'^2(r)}{\chi^2(R)} -3\frac{\chi'' (r)}{C_R}\bigg)  \langle Dr, {\sf v}\rangle^2 \\
			& \le   C\big({\max}_{\bar\Omega} \chi, {\max}_{\bar\Omega} \chi ',   {\max}_{\bar\Omega} \chi '' \big).
		\end{align*}
		Similarly, since $D\xi = \xi'(r) Dr$ one gets
		\begin{align*}
			\frac{1}{\xi} ||D\xi||_{M_s}  = 3 \frac{\chi'(r)}{\varrho(r)}  \le   C\big({\max}_{\bar\Omega} \chi, {\max}_{\bar\Omega} \chi ' \big).
		\end{align*}
		Coming back to \eqref{heat-h} we conclude
		\begin{align}\label{heat-h2}
			Qh + 2\frac{1}{\xi}\langle \nabla \xi, \nabla h\rangle\le  C\big(\chi(R), {\max}_{\bar\Omega} \chi ',   {\max}_{\bar\Omega} \chi '' \big)\mathcal{N} h + \frac{f'(t)}{f(t)} h 
			+ \frac{6}{f\xi}  h^2.
		\end{align}
		Chose $C_4, \,C_5>0$ with $C_4>T$ such that 
		\[
		\frac{C_4}{C_5^ 2}\le \min _{\bar \Omega \times [0, T]}|A|^ 2\xi .
		\]
		In this way the equation \eqref{heat-h2} may be estimated 
		\begin{align*}
			Qh + 2\frac{1}{\xi}\langle \nabla \xi, \nabla h\rangle\le C_6 \big(\chi(R), {\max}_{\bar\Omega} \chi ',   {\max}_{\bar\Omega} \chi '' , \mathcal{N}, R\big )h^ 2.
		\end{align*}
		Since $h$ is supported in $\bar \Omega$ we have $\langle \nabla h, \mu\rangle =0$. Thus the Lemma \ref{smp} provides that 
		\begin{equation}\label{gamma}
		h\le  \gamma(t), \,\, {\rm in} \, \bar \Omega \times [0, T],
		\end{equation}
		where $\gamma$ is a solution of
		\[
		\frac{d\gamma}{\gamma^2} = C_6 dt
		\]
		with $\gamma (0)\geq h (x, 0)$  for all $x\in \bar \Omega $.
		A possible choice for $\gamma $ is 
		\[
		\gamma (t)=-\frac{1}{C_6t+C_7},
		\]
		with $C_7<\min \{-\frac{1}{\max _{\Sigma _0}h }, -C_6T \}.$ Therefore from \eqref{gamma} we get
		\[
		|A|^ 2\le \frac{C_4-T}{C_5(C_6T+C_7)}.
		\]
		This finishes the proof. \end{proof}

	\section{Proof of Theorem \ref{theo2}}\label{proof}
	
	The spatial gradient estimate \eqref{est-conv} in Proposition \ref{est-grad-2} implies that the evolving graphs $\Sigma_t$, $t\in [0, T)$ are strictly spacelike. Moreover, implies that the partial differential equation in \eqref{F2} is uniformly parabolic in $\bar\Omega \times [0, T)$. These  $C^1$ bounds combined with the $C^0$ estimates \eqref{height-1} and \eqref{est-rho} permit to apply the standard theory of quasilinear parabolic PDEs as detailed in [Chapter XIII, \ref{L:96}]. This theory establishes the existence of a smooth solution $u$ of \eqref{F2} defined in $\bar\Omega \times [0, T)$ for an arbitrary $T>0$. 
	
	Now, using Proposition \ref{hess-est-prop} and proceeding as in [\ref{L:14}, Lemma 8.4] one obtains higher order estimates for $u$ (in terms of bounds for higher covariant derivatives of $A$). In particular, since bounds for $H$ depend on estimates for $A$ and its derivatives one gets from \eqref{aga} that
	\[
	\Big|\frac{\partial u}{\partial t}\Big| = \frac{W}{\rho} |H| \le |H|. 
	\]
	Hence, we conclude that the higher order estimates for spatial derivatives of $u$ give bounds for the time derivatives of $u$. These higher estimates imply that there exists a smooth limit of $u$ as $t\to T$. Then the function $u$ can be extended past time $T$ to provide a solution of \eqref{F2} defined in $\bar\Omega \times [0, +\infty)$.
	
	Let us study the asymptotic behavior of the flow. Let	$s_0=\min _{\Sigma_0} s+\epsilon $ with $\epsilon >0$ and $s_1=\max_{\Sigma _0} s+\epsilon$.
	Recall that the mean curvature of a slice $\{s\}\times M$ is given by $H=n\frac{\rho '(s)}{\rho (s)}.$ Hence $\varphi _i: \bar\Omega \times [0, \infty) \to N$ given by
	\begin{eqnarray}\label{mcf-slices}
	\frac{\partial \varphi _i}{\partial t}=n\frac{\rho '}{\rho }\nu , \,\,\,\, i=0, 1
	\end{eqnarray}
	defines a mean curvature flow with initial data $\{s_0\}\times \bar \Omega $ and $\{s_1\}\times \bar \Omega $, respectively. We will denote by $\underline{s}= \pi_\mathbb{R}\circ \varphi _0$ and $\overline{s}= \pi_\mathbb{R}\circ \varphi _1$ the height functions of the flows defined above.
	We have from \eqref{ds} that 
	\begin{eqnarray}\label{s}
	\dfrac{d\overline{ s}}{dt}=n\dfrac{\rho '(\overline{ s})}{\rho (\overline{ s})} \,\,\,\, {\rm and}\,\,\,\,\, \dfrac{d\underline{s}}{dt}=n\dfrac{\rho (\underline{s}) '}{\rho (\underline{s})},
	\end{eqnarray}
	whence 
	\begin{eqnarray}\label{t}
	nt(\underline{s})=\displaystyle \int _{s_0}^ {\underline{s}}\dfrac{\rho (r) }{\rho '(r)}dr \,\,\,\,\,{\rm and}\,\,\,\,\, nt(\overline{s})=\displaystyle \int _{s_1}^ {\overline {s}}\dfrac{\rho (r) }{\rho '(r)}dr .
	\end{eqnarray}
	Let us consider the oscillation function $osc (t)=\overline{s}(t)-\underline{s}(t)\ge 0$ that gives the distance between the slices $\{s_0\}\times \bar \Omega $ and $\{s_1\}\times \bar \Omega  $ as they evolves by the mean curvature flows \eqref{mcf-slices}. Let $t_-, t_+\in [0, \infty )$ with $t_-<t_+$. From \eqref{t} we have that
	\[
	\displaystyle \int _{\overline{s}(t_-)}^ {\overline{s}(t_+)}\dfrac{\rho (r)}{\rho '(r)}dr=\displaystyle \int _{\underline {s}(t_-)}^ {\underline {s}(t_+)}\dfrac{\rho (r)}{\rho '(r)}dr.
	\]
	From the assumption \eqref{max-princ} we get that $\frac{\rho }{\rho '}$ is non-decreasing in $s$. In addition from \eqref{s} and hypothesis \eqref{max-princ} one can check that $\overline{ s}$ and $\underline{s}$ are increasing, hence from above the interval $[\overline{ s}(t_-), \overline{ s}(t_+) ]$ cannot have a length bigger than the interval $[\underline { s}(t_-), \underline { s}(t_+) ]$. Therefore 
	\[
	osc(t_+)=\overline{ s}(t_+)-\underline{ s}(t_+)\le \overline{ s}(t_-)-\underline{ s}(t_-)=osc(t_-),
	\] 
	and $osc$ is non-increasing. In addition, looking to \eqref{t} one checks that the hypothesis \eqref{max-princ} disable that $osc $ to be constant, unless $osc $ is zero. 
	
	
	Now we recall the avoidance principle, which asserts that if two manifolds are initially disjoint and at least one of them is compact, then they remain so during the flow (see \cite{B:16}). 
	Consider $osc^\Sigma (t)=\max _{\Sigma _t}s-\min _{\Sigma _t}s$ the oscillation of a solution of \eqref{Fn}.
	Comparing $\{s_0\}\times \bar \Omega $, $\Sigma _ t $ and $\{s_1\}\times \bar \Omega $ and taking the avoidance principle into account, we accomplish that $osc^ \Sigma $ must be non-increasing. In addition, since $osc $ is non-constant, unless it is zero, then so is $osc ^ \Sigma $.
	
	
	Let $c:=\lim _{t\to \infty }osc^\Sigma (t)$. Define $\psi _ i: \bar \Omega \times [i, \infty )\to N$, $i\in \mathbb{N}$ by $\psi _i(t, x)=\varphi (t+i, x)$. The equation \eqref{matM} provides that the $C^ 1$ estimate \eqref{est-conv} holds true for $\psi _ i$ as well. From [\ref{L:96}, Theorem 12.1] we obtain uniform Holder estimates for the flows $\psi _ i$ and from Schauder theory ref.....Arzela-Ascoli... we obtain the smooth limit mean curvature flow 	
	\[
	\varphi _\infty =\lim _{i\to \infty }\psi _ i.
	\]
	By construction, the flow $\varphi _\infty $ has constant oscillation equal to $c$ and, up to the initial data, $\varphi _ \infty $ solves \eqref{Fn} as well. As we saw above, we must have $c=0$.
	
	Recall that the metric associated to a solution of \eqref{Fn} is given by 
	\begin{center}
		$\varphi (\cdot, t)^ * \langle \,, \, \rangle $	with $\langle \, ,\, \rangle =	- \pi_\mathbb{R}^ * {\rm d}s^2 + \rho^2(\pi _\mathbb{R}) \pi _M^ * \sigma$,
	\end{center}
	in each time $t$. Since $\max _{\Sigma _t}s-\min _{\Sigma _t}s=osc^ \Sigma (t)\to 0$ as $t\to \infty$,
	then
	\[
	\lim _{t\to \infty }\varphi (\cdot, t)^ *\langle \, ,\, \rangle =\rho (\pi _\mathbb{R})^ 2\pi _M^* \sigma ,
	\]
	in which case it is conformal to the metric $\sigma $ of the leaf $M$.

	Now let us prove that the mean convexity is preserved. Taking the equation \eqref{riccii} into account we check that $\overline{\rm Ric }(\partial _s, \partial _s)=-n\frac{\rho ''}{\rho }$. In this case the boundedness over $\overline{\rm Ric }(\partial _s, \partial _s)$ is translated in the condition \eqref{curv-hip}, and we conclude from the Proposition \ref{mean-convex} that if the initial graph $\Sigma _ 0$ is mean convex, then the evolving graphs remain mean convex.

	\section{Appendix: A Simons' type identity}
	
	\begin{proof}{\it of Proposition \ref{simons}}
		Let $\overline{ R}$ and $ R$ be the Riemann tensors of $N$ and $\Sigma _t$ defined by:
		\[
		\overline{ R}(X, Y)Z=\overline{\nabla }_X\overline{\nabla }_YZ-\overline{\nabla }_Y\overline{\nabla }_XZ-\overline{\nabla }_{[X, Y]}Z,
		\]
		for all $X, Y, Z\in TN$ and
		\[
		R(U, V)W=\nabla _U\nabla _VW-\nabla _V\nabla _UW-\nabla _{[U, V]}W,
		\]
		for all $U, V, W\in T\Sigma _t$.
		
		Set the coordinate vector fields $\{\partial_i\}_{i=1}^{n+1}$ in $N$. In what follows we use Codazzi's equation
		\begin{align}
			\label{codazzi-cod-1}  \nabla_i a_{jk} -\nabla_j a_{ik}&= \partial_i \langle \bar\nabla_{\partial_j} \nu, \partial_k\rangle - \partial_j \langle \bar\nabla_{\partial_i}\nu, \partial_k\rangle - \langle A(\nabla_{\partial_i}\partial_j), \partial_k\rangle - \langle A\partial_j, \nabla_{\partial_i}\partial_k\rangle\nonumber \\
			&\,\,\,\,\,\,\,\,+ \langle A(\nabla_{\partial_j}\partial_i), \partial_k\rangle +\langle A\partial_i, \nabla_{\partial_j}\partial_k\rangle =- L_{ijk},.
		\end{align}
		being $L$ the $(0,3)$-tensor in the hypersurface $\Sigma _t$ given by
		\begin{equation}
		\label{L} L_{ijk}=\langle \bar R(\partial_i,\partial_j)\nu ,
		\partial_k\rangle.
		\end{equation}
		We also have need of (Riemannian) Ricci commutation formula
		\begin{align}
			\label{ricci-commutation} & \nabla_i \nabla_k
			a_{j\ell}-\nabla_k\nabla_i a_{j\ell}= 
			-R^s_{ikj}
			a_{s\ell}-R^s_{ik\ell} a_{js}.
		\end{align}
		Using (\ref{codazzi-cod-1}) and (\ref{ricci-commutation}) one
		obtains
		\begin{eqnarray*}
			\nabla_k \nabla_\ell a_{ij} &=& \nabla_k \nabla_i a_{\ell j}
			-\nabla_k L_{\ell ij}\\
			&=&\nabla_i \nabla_k a_{\ell j} - R^s_{ki\ell} a_{sj}-
			R^s_{kij} a_{\ell s} -\nabla_k L_{\ell ij}\\
			&=& \nabla_i \nabla_j a_{k\ell } -\nabla_i
			L_{k\ell j}-\nabla_k L_{\ell ij}- R^s_{ki\ell} a_{sj}-R^s_{kij}
			a_{\ell s}.
		\end{eqnarray*}
		Now we should take in account the Gauss equation
		\begin{equation*}
			R^s_{ik\ell} = \bar R_{ik\ell}^s  +a_{k\ell} a_i^s - a_{i\ell}
			a^s_k.
		\end{equation*}
		Therefore
		\begin{eqnarray*}
			\nabla_k \nabla_\ell a_{ij}&=&\nabla_i \nabla_j a_{k\ell }
			-\nabla_i L_{k\ell j}-\nabla_k L_{\ell ij}\\
			& &- (\bar R^s_{ki\ell}+a_{i\ell} a^s_k - a_{k\ell} a^s_i)
			a_{sj}- (\bar R^s_{kij} +a_{ij}a^s_k -a_{jk}a^s_i) a_{\ell s}.
		\end{eqnarray*}
		Taking traces we have
		\begin{eqnarray*}
			\Delta a_{ij} &=&\nabla_i \nabla_j H
			-g^{k\ell}(\nabla_i L_{k\ell j}+\nabla_k L_{\ell ij}) - g^{k\ell}(\bar R^s_{ki\ell}a_{sj}+\bar R^s_{kij}a_{\ell s})\\
			& &- a^k_{i} a^s_k a_{sj}+H a^s_i a_{sj}- a_{ij}|A|^2
			+a^\ell_{j}a^s_i a_{\ell s}.
		\end{eqnarray*}
		However
		\[
		a^\ell_{j}a^s_i a_{\ell s}=a^s_{j}a^\ell_i a_{s \ell}=a^s_{j}a^k_i
		a_{s k}= a^k_i a_{sj} a^s_{k}.
		\]
		This implies the following general formula
		\begin{eqnarray*}
			\Delta a_{ij} &=&\nabla_i \nabla_j H +H a^s_i a_{sj}-
			a_{ij}|A|^2\nonumber\\
			& & -g^{k\ell}(\nabla_i L_{k\ell j}+\nabla_k L_{\ell ij}) -
			g^{k\ell}(\bar R^s_{ki\ell}a_{sj}+\bar R^s_{kij}a_{\ell s}).
			\label{delta-aij-general}
		\end{eqnarray*}
		Hence,
		\begin{eqnarray*}
			\frac{1}{2} \Delta |A|^2 -|\nabla A|^2  & =& a^{ij}\Delta a_{ij}\\
			& =&a^{ij}\nabla_i \nabla_j H +H \operatorname{tr} A^3-
			|A|^4\nonumber -g^{k\ell}(\nabla_i L_{k\ell j}+\nabla_k L_{\ell ij}) a^{ij} \\
			& &-g^{k\ell}(\bar R^s_{ki\ell}a_{sj}a^{ij}+\bar R^s_{kij}a_{\ell s} a^{ij}).
		\end{eqnarray*}
		On the other hand
		\begin{align*}
			\partial_t a_{ij}& = \langle \bar\nabla_{\partial_t} \bar\nabla_{\partial_i} \nu, \partial_j\rangle + \langle\bar\nabla_{\partial_i} \nu, \bar\nabla_{\partial_j} \partial_t\rangle  \\
			& = \langle \bar\nabla_{\partial_i} \bar\nabla_{\partial_t} \nu, \partial_j\rangle + \langle \bar R(\partial_t, \partial_i)\nu, \partial_j\rangle + \langle\bar\nabla_{\partial_i} \nu, \bar\nabla_{\partial_j} \partial_t\rangle \\
			&= \langle \nabla_{\partial_i} \nabla H, \partial_j\rangle  - H\langle \bar R(\partial_i,\nu)\nu, \partial_j\rangle + \langle\bar\nabla_{\partial_i} \nu, \bar\nabla_{\partial_j} (H\nu)\rangle \\
			& = \langle \nabla_{\partial_i} \nabla H, \partial_j\rangle  - H\langle \bar R(\partial_i,\nu)\nu, \partial_j\rangle + H a_{i}^k a_{jk}.
		\end{align*}
		Therefore
		\begin{align*}
			\frac{1}{2}\partial_t |A|^2  = a^{ij}\langle \nabla_{\partial_i} \nabla H, \partial_j\rangle  - Ha^{ij}\langle \bar R(\partial_i,\nu)\nu, \partial_j\rangle + H a_{i}^k a_{jk} a^{ij}.
		\end{align*}
		We conclude that
		\begin{eqnarray*}
			\frac{1}{2} Q |A|^2 +|\nabla A|^2  & =&
			|A|^4+g^{k\ell}(\nabla_i L_{k\ell j}+\nabla_k L_{\ell ij}) a^{ij} +
			g^{k\ell}(\bar R^s_{ki\ell}a_{sj}a^{ij}+\bar R^s_{kij}a_{\ell s} a^{ij})\\
			& & -H\langle \overline{ R}(\partial _ i, \nu )\nu , \partial _ j\rangle .
		\end{eqnarray*}
		Therefore
		\begin{eqnarray*}
			\frac{1}{2} Q|A|^2 +|\nabla A|^2  & = & |A|^4+T_{ij} a^{ij} +
			g^{k\ell}(\bar R^s_{ki\ell}a_{sj}a^{ij}+\bar R^s_{kij}a_{\ell s} a^{ij}) -H\langle \overline{ R}(\partial _ i, \nu )\nu , \partial _ j\rangle ,
		\end{eqnarray*}
		where $T_{ij}=g^{k\ell}(\nabla_i L_{k\ell j}+\nabla_k L_{\ell ij}) a^{ij}.$

	\end{proof}


\begin{thebibliography}{99} 
		\bibitem{AMR:16}\label{AMR:16} 
		L. Al\'ias, P. Mastrolia and M. Rigoli. 
		{\em Maximum principles and geometric applications}. 
		Springer Monographs in Mathematics. Springer, Cham, (2016), xvii+570 pp. 
		
		\bibitem{ARS:95}\label{ARS:95} 
		L. Al\'ias, A. Romero, M. S\'anchez. 
		{\em Uniqueness of complete spacelike hypersurfaces of constant mean curvature in Generalized Robertson Walker spacetimes}. 
		T\^ohoku Mathematical Journal \textbf{ 27} (1995) 71--84. 
		
		\bibitem{ARS:97}\label{ARS:97} 
		L. Al\'ias, A. Romero, M. S\'anchez. 
		{\em Spacelike hypersurfaces of constant mean curvature and Calabi-Berstein type problems}. 
		T\^ohoku Mathematical Journal \textbf{ 30} (1997) 655--661. 
		
		\bibitem{B:16}\label{B:16} 
		G. Bellettini. 
		{\em Lecture notes on mean curvature flow: barriers and singular perturbations}, 
		Bulletin (New Series) of the American Mathematical Society \textbf{ 54} (2016) 529--532. 
		
		\bibitem{BBC:75}\label{BBC:75} 
		A. Barros, A. Brasil, A. Caminha. 
		{\em Stability of spacelike hypersurfaces in foliated spacetimes}, 
		Differential Geometry and its Applications \textbf{ 26} (2008) 357--365. 
		
		\bibitem{CK:04}\label{CK:04} 
		B. Chow, D. Knopf. 
		{\em The Ricci flow: an introduction}. 
		Mathematical Surveys and Monographs, American Mathematical Society {\bf 110} (2004). 
		
		%
		
		\bibitem{E:93}\label{E:93} 
		K. Ecker. \textit{On the mean curvature flow of spacelike hypersurfaces in asymptotically flat spacetimes}. Journal of the Australian Mathematical Society, {\bf 55} (1993) 41-59. 
		
		\bibitem{E:97}\label{E:97} 
		K. Ecker. \textit{Interior estimates and longtime solutions for mean curvature flow of noncompact spacelike hypersurfaces in Minkowski space}. Journal of Differential Geometry, {\bf 45} (1997) 481-498. 
		
		\bibitem{EH:91}\label{EH:91} K. Ecker and G. Huisken. \textit{Parabolic methods for the construction of spacelike slices of prescribed mean curvature in cosmological spacetimes}. Communications in Mathematical Physics, {\bf 135} (1991) 595-613. 
		
		\bibitem{GT:01}\label{GT:01} 
		D. GIlbarg and N. S. Trudinger \textit{Elliptic Partial Differential Equations of Second Order}, Springer-Verlag Berlin Heidelberg, (2001)
		
		
		\bibitem{H:11}\label{H:11} 
		Q. Han \textit{Flow by mean curvature of convex surfaces into spheresA Basic Course in Partial Differential Equations}, American Mathematical Soc., (2011)
		
		\bibitem{H:84}\label{H:84} 
		G. Huisken \textit{Flow by mean curvature of convex surfaces into spheres}, Journal of Differential Geometry {\bf 20} (1984) 237-266. 
		
		\bibitem{L:14}\label{L:14} 
		B. Lambert. \textit{The perpendicular Neumann problem for mean curvature flow with a timelike cone boundary condition}. Transactions of the American Mathematical Society, {\bf 7} (2014) 3373-3388. 
		
		\bibitem{L:96}\label{L:96} 
		G. Lieberman. \textit{Second order parabolic differential equations}. World Scientific Publishing Co. Pte. Ltd., River Edge, NJ (1996)
		
		\bibitem{M:73}\label{M:73} 
		C. W. Misner, K. S. Thorne, and J. A. Wheeler. \textit{Gravitation}. W. H. Freeman and Co., New York,
		(1973).
		
		\bibitem{M:99}\label{M:99} 
		S. Montiel. {\em Uniqueness of spacelike hypersurfaces of constant mean curvature in oliated spacetimes}. Proceedings of the American Mathematical Society, {\bf 3} (1999) 529-533. 
		
		
		\bibitem{O'N:83}\label{O'N:83} 
		B. O'Neill. 
		{\em Semi-Riemannian Geometry with Applications to Relativity} 
		Academic Press, London (1983). 
		
		\bibitem{RRS:13}\label{RRS:13}
		{A. Romero, R. M. Rubio \and J. J. Salamanca},
		{\it Uniqueness of complete maximal hypersurfaces in spatially parabolic generalized Robertson-Walker spacetimes,},
		Classical Quantum Gravity {\bf 30} (2013) 115007 pp. 13.
	\end{thebibliography}
	\end{document}